\documentclass[11pt]{amsart}

\usepackage{amsmath, amsthm, amscd, amsfonts, amssymb, enumerate, verbatim, newlfont, calc, graphicx, color}
\usepackage[bookmarksnumbered, colorlinks, plainpages]{hyperref}
\usepackage{tikz}
 \usepackage{doi}

\newtheorem{theorem}{Theorem}[section]

\newtheorem{lemma}[theorem]{Lemma}
\newtheorem{proposition}[theorem]{Proposition}
\newtheorem{conjecture}[theorem]{Conjecture}
\newtheorem{corollary}[theorem]{Corollary}
\newtheorem{definition}[theorem]{Definition}

\newtheorem{notation}[theorem]{Notation}
\newtheorem{remark}[theorem]{Remark}

\newtheorem{example}[theorem]{Example}

\newtheorem{fact}[theorem]{Fact}

\usepackage{mathtools}
\usepackage{pstricks}
\usepackage{epsfig}
\usepackage{pst-grad}
\usepackage{pst-plot}
\usepackage{subfig}

\begin{document}
\author[A.  Bretto,  M. Nasernejad,   and   J. Toledo]{Alain ~Bretto, Mehrdad ~Nasernejad$^{*}$,   Jonathan Toledo}
\title[Strong persistence property and normally torsion-freeness] {On the strong persistence property and normally torsion-freeness of square-free  monomial ideals}

\thanks{$^*$Corresponding author}


\maketitle


\begin{abstract}
In this paper,  we first show that any square-free monomial ideal in $K[x_1, x_2, x_3, x_4, x_5]$ has  the   strong persistence property. Next, 
we will provide a criterion for a minimal counterexample to the Conforti-Cornu\'e{j}ols conjecture.  Finally, we give 
 a necessary and sufficient condition to determine the normally torsion-freeness  of a linear combination of two  
  normally torsion-free square-free monomial  ideals.
\end{abstract}
\vspace{0.4cm}


 \section{Introduction}

Let $I$ be an ideal in a  commutative  Noetherian ring  $R$. A prime ideal $\mathfrak{p}\subset  R$ is an {\it associated prime} of $I$ if there exists an element $v$ in $R$ such that $\mathfrak{p}=(I:_R v)$. The  {\it set of associated primes} of $I$, denoted by  $\mathrm{Ass}_R(R/I)$, is the set of all prime ideals associated to  $I$. We will be interested in the sets $\mathrm{Ass}_R(R/I^s)$ when $s$  varies. A notable  result of  Brodmann \cite{BR} showed that the sequence 
$\{\mathrm{Ass}_R(R/I^s)\}_{s \geq 1}$ of associated prime ideals is stationary  for large $s$. That is to say, there exists a positive integer $s_0$ such that $\mathrm{Ass}_R(R/I^s)=\mathrm{Ass}_R(R/I^{s_0})$ for all $s\geq s_0$. The  minimal such $s_0$ is called the
{\it index of stability}   of  $I$ and $\mathrm{Ass}_R(R/I^{s_0})$ is called the {\it stable set } 
 of associated prime ideals of  $I$, which is denoted by $\mathrm{Ass}^{\infty }(I).$
  There have been many  questions arising from Brodmann’s result.
 One of the questions is as follows:   Is it true that  $\mathrm{Ass}_R(R/I^s)\subseteq \mathrm{Ass}_R(R/I^{s+1})$ for all $s\geq 1$?
 According to an example of McAdam \cite{Mc}, in general, this   question has a negative answer. We say that an ideal $I$ of $R$  satisfies the {\it persistence property} if 
 $\mathrm{Ass}_R(R/I^s)\subseteq \mathrm{Ass}_R(R/I^{s+1})$ for all  $s\geq 1$. 
In addition,  an ideal  $I$ of $R$ has the {\it strong persistence property}  if $(I^{s+1}:_R I)=I^s$ for all  $s\geq 1$. However, Ratliff \cite{RA} proved that   $(I^{s+1}:_R I)=I^s$ for all large $s$. In particular, it should be noted that  the strong  persistence property implies the persistence property (for example see  the proof of \cite[Proposition 2.9]{N1}), but  the converse is not true.

Suppose  now  that $I$ is a monomial  ideal  in  a  polynomial ring $R=K[x_1,\ldots,x_n]$ over a field $K$ and $x_1,\ldots,x_n$ are indeterminates. 
The above question  does not have an affirmative  answer even for all square-free   monomial ideals, such a counterexample can be found in  \cite{KSS}.
Monomial ideals have a  profound  role in exploring  the connection between combinatorics and  commutative algebra such that  the relation between these two  areas  enables  us to employ techniques   in commutative algebra to dig    combinatorial problems, and vice versa. For this reason,  commutative algebraists 
have   commenced  studying   the properties of   combinatorial objects such as  graphs, hypergraphs, simplicial complexes, and posets 
 through monomial ideals. One of the pioneers   in this field is  Villarreal  who  introduced the notion  of edge ideals.  In fact, the {\it edge ideal} corresponding to  a simple finite undirected  graph $G$ has been introduced in \cite{VI1} and is generated by  the monomials $x_ix_j$, where $\{i, j\}$ is an
edge of $G$, and the Alexander dual of  it, denoted by  $I(G)^\vee$, is called the {\it cover ideal} corresponding to $G$.

  Up to now,  by applying combinatorial approaches,  a bunch of papers  have been published for detecting  the classes of monomial ideals which satisfy the persistence property and the strong persistence property.  Particularly, the following classes of monomial ideals satisfy the strong  persistence property:
\begin{itemize}
\item[(1)] $I$ is the  edge ideal of a simple graph (see \cite[Theorem 2.15]{MMV}).
\item[(2)]  $I$ is a  polymatroidal ideal  (see \cite[Proposition 2.4]{HQ}). 
\item[(3)] $I$ is the edge ideal  of a finite   graph with  loops  (see \cite[Corollary 2]{RT}).
\item[(4)] $I$ is a unisplit (or a  separable) monomial ideal (see  \cite{N2}). 
\end{itemize}
On the other hand,  according to   \cite[Theorem 6.2]{RNA}, every normal monomial ideal has the strong persistence property. Hence, 
 the cover ideals of the following   graphs  have  the strong persistence property: 
\begin{itemize}
\item[(1)]  Cycle graphs of odd orders (see \cite[Theorem 3.3]{NKA}).
\item[(2)] Wheel  graphs of even orders  (see \cite[Theorem 3.9]{NKA}).
\item[(3)] Helm  graphs of odd orders at least $5$  (see \cite[Theorem 3.12]{NKA}).
\item[(4)]  Jahangir’s graphs (see \cite[Corollary 2.5]{ANR1}).
\item[(5)]  Theta graphs (see \cite[Theorem 2.11]{ANR2}).
\item[(6)] The union of two graphs (see \cite[Lemma 6 and Theorem 11]{KNT}).
\end{itemize}
Along this argument, it has already been stated in \cite[Theorem 9]{RT} that any square-free monomial ideal in $K[x_1, x_2, x_3, x_4]$  has the strong
persistence property. Nevertheless,  the ideal $I\subset K[x_1, x_2, x_3, x_4, x_5, x_6]$ generated by 
$x_1x_2x_3,$ $x_1x_2x_4,$ $x_1x_3x_5,$ $x_1x_4x_6,$ $x_1x_5x_6,$ $x_2x_3x_6,$ $x_2x_4x_5,$ $x_2x_5x_6,$
 $x_3x_4x_5,$ $x_3x_4x_6$ does not satisfy the strong persistence property due to $(I^3:I)\neq I^2$. So,  it is natural to ask that is any square-free monomial ideal in $K[x_1, x_2, x_3, x_4,x_5]$  has the strong persistence property? The first motivation of this paper is to provide an affirmative answer to this question. 
 
 We now turn our attention to the second aim of this paper. An  ideal $I$ in a commutative Noetherian ring $R$ is called {\it normally torsion-free} 
 if $\mathrm{Ass}_R(R/I^s)\subseteq \mathrm{Ass}_R(R/I)$   for all $s\geq 1$. So far,  little is known for the normally torsion-free monomial
ideals.   It  has been proved   that  a finite simple graph $G$ is bipartite if and only if its edge ideal is normally torsion-free,  if and only if its cover ideal is normally torsion-free, refer to  \cite{GRV, SVV}. Moreover,  in \cite[Theorem  3.3]{N3}, the authors showed   that  the Alexander dual of the monomial
ideal generated by the paths of maximal lengths in  a rooted starlike tree is normally torsion-free. 
Also, in  \cite[Theorem 3.2]{KHN1}, it has been proved  that   the Alexander dual of path ideals generated by all paths of length $2$ in rooted trees are normally
torsion-free. More recently, it has been established in \cite[Theorem 2.8]{NQ} that if $G$ is  a strongly chordal graph, then  both $NI(G)$ and $DI(G)$ are  normally torsion-free, where $NI(G)$ (respectively, $DI(G)$) stands for the closed neighborhood ideal of $G$ (respectively, dominating ideal of $G$). 
Moreover, based on  \cite[Theorem 4.3]{NQ},    $DI(C_n)$ is normally torsion-free if and only if $n \in \{3,6,9\}$, 
the definitions of closed neighborhood ideals and  dominating ideals and their properties can be found in \cite{NBR, NQBM}. 
 Furthermore,  \cite[Lemma 5.15]{NQKR} and   \cite[Theorem 4.6]{MNQ} give classes of normally torsion-free $t$-spread monomial ideals.
   In particular, \cite[Theorem 3.11]{SNQ} says that every square-free transversal polymatroidal ideal is normally torsion-free.
 
 In broadly speaking,  normally torsion-free square-free monomial ideals are closely related to Mengerian hypergraphs.  A hypergraph  is called 
\textit{Mengerian} if it satisfies a certain min-max equation, which is known as the Mengerian property in hypergraph theory or has the max-flow min-cut property in integer programming. In fact, based on  \cite[Theorem 14.3.6]{V1},  it is well-known that if $\mathcal{C}$ is  a clutter and $I=I(\mathcal{C})$  its edge ideal, then $I$ is  normally torsion-free if and only if  $I^{k}=I^{(k)}$ for all $k\geq 1$, if and only if $\mathcal{C}$ is Mengerian, if and only if $\mathcal{C}$  has the max-flow min-cut property. One of the famous open problems in this theme is the Conforti-Cornu\'e{j}ols conjecture. The second motivation of this paper comes back 
 to provide a criterion for a minimal counterexample to the Conforti-Cornu\'e{j}ols conjecture. 
 
We are now ready to discuss the  final target of this  work. The third motivation of this paper originates from two results related to the integrally closedness  (resp.   normality)  of  a linear combination of two  integrally closed (resp.  normal) (square-free) monomial ideals. Indeed, let $L=x_iI+x_jJ\subset R=K[x_1, \ldots, x_n]$ with $1\leq i\neq j \leq n$ be a square-free monomial ideal such that $\mathrm{gcd}(x_j, u)=1$  and   $\mathrm{gcd}(x_i, v)=1$   for all $u\in \mathcal{G}(I)$ and  $v\in \mathcal{G}(J)$. Already, in \cite{ANKRQ} and  \cite{ANR}, it has been argued on the integrally closedness and normality of $L$ when $I$ and $J$ are integrally closed (resp. normal) (square-free) monomial ideals.  Inspired by these results, we are interested in the normally torsion-freeness of $L$ when $I$ and $J$ are normally torsion-free. 

 This paper is organized as follows.  In Section 2, we collect all of the necessary definitions and results which will be used in the rest of this paper. 
Section 3  is devoted to showing  that any square-free monomial ideal in $K[x_1, x_2, x_3, x_4, x_5]$ has  the   strong persistence property,
cf.  Theorem  \ref{main.Section3}. Section 4  deals with  providing  a criterion for a minimal counterexample to the Conforti-Cornu\'e{j}ols conjecture, 
see Corollary   \ref{Cor.3}.   Finally, Section 5  is concerned with giving  a necessary and sufficient condition to determine the normally torsion-freeness  of a linear combination of two     normally torsion-free square-free monomial  ideals, refer to Theorem \ref{Linear-Combination}.


\section{Preliminaries}

We here collect the  necessary definitions and results which will be used in the rest of this paper. We begin with the following proposition. 

\begin{proposition} (\cite[Proposition 8]{RT}) \label{RT-Pro-8}
A clutter $\mathcal{C}$  with $3$ edges has the strong persistence property.
\end{proposition}


\begin{lemma} (\cite[Lemma 2.2]{NKRT}) \label{NKRT-Lem-2.2}
Suppose that $I_1$ and $I_2$ are two monomial ideals in a polynomial ring $R = K[x_1, \ldots, x_n]$ over a field $K$ 
 such that $\mathrm{supp}(I_1) \cap \mathrm{supp}(I_2) = \emptyset$. Then  $I_1+I_2$ has the strong persistence property 
if and only if $I_1$ or $I_2$ has the strong persistence property.
\end{lemma}


\begin{theorem} (\cite[Theorem 9]{RT}) \label{RT-Th-9}
If $I$ is a squarefree monomial ideal in  $K[x_1,x_2,x_3,x_4]$, then $I$ has the strong persistence property.
\end{theorem}


\begin{lemma} (\cite[Lemma 2]{RT}) \label{RT-Lem-2}
 Let $\mathcal{C}$  be a clutter. If there exists an edge $f \in  E(\mathcal{C})$ such that  $A =\{g\cap  f \mid  g \in E(\mathcal{C})\}$ 
 is a chain, then $I(\mathcal{C})$  has the strong persistence property.
\end{lemma}


\begin{definition} (\cite[Page 303]{HM}) 
\em{
 Assume that $I\subset R=K[x_1, \ldots, x_n]$ is  a square-free monomial ideal and $\Gamma \subseteq \mathcal{G}(I)$, where $\mathcal{G}(I)$ denotes the unique minimal set of monomial generators of  $I$. Then we say $\Gamma$ is an {\it independent} set in $I$ if $\mathrm{gcd}(f,g)=1$ for each $f,g\in \Gamma$ with $f\neq g$. We denote the maximum cardinality of an independent set in $I$ by $\beta_1(I)$. 
}
\end{definition}


  \begin{definition} (\cite[Definition 6.5.2]{V1})
  \em{
Let  $I\subset R=K[x_1, \ldots, x_n]$ be  a square-free monomial ideal. Then the  {\it deletion} (respectively, \textit{contraction}) of $x_i$  from  $I$ with $1\leq i \leq n$, denoted by $I\setminus x_i$ (respectively, $I/x_i$), is obtained by setting $x_i=0$ (respectively, $x_i=1$) in every minimal generator of $I$.
}
\end{definition}
 

\begin{definition} (\cite[Definition 6.1.5]{V1})
\em{
The \textit{support} of a monomial $u$ in $R=K[x_1, \ldots, x_n]$, denoted by $\mathrm{supp}(u)$, is the set of variables that divide $u$. In particular, we consider 
$\mathrm{supp}(1)=\emptyset$.  Moreover, for a monomial ideal $I$ of $R$, we set $\mathrm{supp}(I)=\bigcup_{u \in \mathcal{G}(I)}\mathrm{supp}(u)$. 
}
\end{definition}


\begin{definition} (\cite[Definition 4.3.22]{V1}) \label{Def.4.3.22}
\em 
{
 Let $I$ be an ideal of a ring $R$ and $\mathfrak{p}_1, \ldots, \mathfrak{p}_r$  the minimal primes of $I$. Given an integer 
 $n \geq 1$,  the $n$th \textit{symbolic power} of $I$ is defined to be the ideal
 $I^{(n)} = \mathfrak{q}_1 \cap \cdots \cap \mathfrak{q}_r$,  where $\mathfrak{q}_i$  is the primary component of $I^n$ corresponding to $\mathfrak{p}_i$. 
 }
 \end{definition}


\begin{fact} (\cite[Exercise 6.1.25]{V1})\label{Ex.6.1.25}
 If  $\mathfrak{q}_1, \ldots, \mathfrak{q}_r$  are primary monomial ideals of $R$ with non-comparable radicals and $I$ 
is an ideal such that $I =\mathfrak{q}_1 \cap  \cdots \cap  \mathfrak{q}_r$, then $I^{(n)} =\mathfrak{q}^n_1 \cap  \cdots \cap  \mathfrak{q}^n_r$. 
\end{fact}


\begin{proposition}   (\cite[Proposition  2.1]{NQT})    \label{Pro.1}
Let $I\subset R=K[x_1, \ldots, x_n]$ be a  monomial ideal,  $t$ a positive integer, $\mathfrak{p}$ a prime monomial ideal in $R$,
  and $y_1, \ldots, y_s$ be distinct variables in $R$  such that, 
for each  $i=1, \ldots, s$, $\mathfrak{p}\setminus y_i \notin \mathrm{Ass}(R/(I\setminus y_i)^t)$, where $I\setminus y_i$ denotes the 
deletion  of $y_i$  from  $I$.  Then  $\mathfrak{p}\in \mathrm{Ass}(R/I^t)$ if and only if $\mathfrak{p}\in \mathrm{Ass}(R/(I^t:\prod_{i=1}^sy_i))$. 
\end{proposition}


\begin{proposition}     (\cite[Proposition  2.2]{NQT})    \label{Pro.2}
  Let   $I \subset R$ be a   monomial ideal, $\ell$ a positive integer, $v$ a square-free monomial in $R$ with $v\in I^\ell$,  and  $\mathfrak{p}$ a prime monomial ideal  such that  $\mathfrak{p}\setminus x_i \notin \mathrm{Ass}(R/(I\setminus x_i)^s)$ for all   $x_i\in \mathrm{supp}(v)$ and some positive integer $s$.  If $\mathfrak{p}\in \mathrm{Ass}(R/I^s)$, then $s>\ell$. 
 \end{proposition}


\begin{theorem} (\cite[Theorem 14.3.6]{V1})  \label{Villarreal1}
 Let $\mathcal{C}$  be a clutter and let  $A$ be its incidence matrix. The following are equivalent:
 \begin{itemize}
 \item[(i)] $\mathrm{gr}_I(R)$ is reduced, where $I=I(\mathcal{C})$ is the edge ideal of $\mathcal{C}$. 
 \item[(ii)] $R[It]$ is normal  and $\mathcal{Q}(A)$ is an integral polyhedron. 
 \item[(iii)]  $x\geq 0$; $xA\geq \textbf{1}$   is a $\mathrm{TDI}$ system. 
 \item[(iv)]  $\mathcal{C}$  has the max-flow min-cut (MFMC) property.
 \item[(v)]  $I^i = I^{(i)}$ for $i \geq 1$. 
 \item[(vi)]  $I$ is normally torsion-free, i.e.,  $\mathrm{Ass}_R(R/I^i) \subseteq  \mathrm{Ass}_R(R/I)$ for $i\geq 1$. 
 \item[(vii)]  $\mathcal{C}$  is Mengerian, i.e., $\beta_1({\mathcal{C}}^a) = \alpha_0({\mathcal{C}}^a)$ for all  $a \in \mathbb{N}^n$.
 \end{itemize}
\end{theorem}


\begin{lemma}(\cite[Lemma 2.1]{KHN2})   \label{Kaplansky}
Let $S$ be a commutative ring and let $a_1,\ldots,a_m$ be elements constituting a permutable  $S$-sequence. Let $J$
 be an ideal generated by monomials in $a_{t+1},\ldots,a_m$  for some $t\in\mathbb{N}$ with $1\leq t \leq m-1$. Then $(J:_S a_1^{n_1} a_2^{n_2}
 \cdots a_t^{n_t})=J$ for all $n_1,n_2,\ldots,n_t\in\mathbb{N}$.
\end{lemma}


 \begin{fact} \label{fact1} (\cite[Exercise 6.1.23]{V1}) 
If $I$, $J$, $L$ are monomial ideals, then  the following equalities hold:
\begin{itemize}
\item[(i)]  $I\cap (J+L)=(I \cap J) + (I \cap L).$
\item[(ii)] $I+ (J \cap L)= (I+J) \cap (I+L).$
\end{itemize}
\end{fact}


\begin{theorem} (\cite[Theorem 2.5]{SN})  \label{NTF.Th.2.5}
Let $I$ be  a  monomial ideal  in  $R=K[x_1, \ldots, x_n]$ such that 
$I=I_1R + I_2R$, where
 $\mathcal{G}(I_1) \subset R_1=K[x_1, \ldots, x_m]$ and $\mathcal{G}(I_2) \subset R_2=K[x_{m+1}, \ldots, x_n]$ for some  positive integer $m$. If $I_1$ and   $I_2$  are normally torsion-free, then  $I$  is so.
 \end{theorem}


\begin{lemma} (\cite[Lemma 3.12]{SN})  \label{NTF.Lem.3.12}
Let  $I$ be a monomial ideal in a polynomial ring $R=K[x_1, \ldots, x_n]$ with 
 $\mathcal{G}(I)=\{u_1, \ldots, u_m\}$, and $h=x_{j_1}^{b_1}\cdots x_{j_s}^{b_s}$ with $j_1, \ldots, j_s \in \{1, \ldots, n\}$ be a monomial in $R$. 
 Then  $I$ is normally torsion-free  if and only if $hI$ is normally torsion-free. 
\end{lemma}


\begin{theorem} (\cite[Theorem 3.19]{SN})    \label{NTF.Th.3.19}
 Let $I$ be  a    monomial ideal in  $R=K[x_1, \ldots, x_n]$, and  
   $1\leq j \leq n$.  If $I$  is normally torsion-free, then    $I/x_j$ is so. 
 \end{theorem}


 \begin{theorem} (\cite[Theorem 3.21]{SN})   \label{NTF.Th.3.21}
 Let $I$ be  a  square-free  monomial ideal in  $R=K[x_1, \ldots, x_n]$, and  
  $1\leq j \leq n$.  If $I$  is normally torsion-free, then    $I\setminus x_j$ is so. 
 \end{theorem}


We close this section by   recalling the needed definitions  related to  hypergraph theory, which can be found in \cite{Berge, V1}. 

A finite {\it hypergraph} $\mathcal{H}$ on a vertex set $V({\mathcal{H}})=\{x_1,x_2,\ldots,x_n\}$ is a collection of  edges  $E({\mathcal{H}})=\{ E_1, \ldots, E_m\}$ with $E_i \subseteq V({\mathcal{H}})$
 for all $i=1, \ldots,m$.  
A hypergraph $\mathcal{H}$ is called {\it simple}  if $E_i \subseteq  E_j$ implies  $i = j$. Simple hypergraphs are also known as {\em clutters}. 
In addition,  if $|E_i|=d$  for all $i=1, \ldots, m$, then $\mathcal{H}$ is called a {\em $d$-uniform} hypergraph. A $2$-uniform hypergraph $\mathcal{H}$ is just a finite simple graph.

Let $V(\mathcal{H})=\{x_1,x_2,\ldots, x_n\}$ and $R=K[x_1, \ldots, x_n]$ be the polynomial ring in $n$ variables  over a field $K$. The {\it edge ideal} of $\mathcal{H}$ is given by
$$I(\mathcal{H}) = \left(\prod_{x_j\in E_i} x_j : E_i\in  E({\mathcal{H}})\right).$$


\section{Square-free monomial  ideals in  $K[x_1, x_2, x_3, x_4, x_5]$ and the  strong persistence property}

In this section we show that  any square-free monomial ideal in the polynomial ring $K[x_1, x_2, x_3, x_4, x_5]$ over a field $K$ has  the   strong persistence property.  To do this, we start  with the following definitions and notation.

\begin{definition}
\em{
     Given a clutter $\mathcal{C}$, we say that $\mathcal{C}$ is a \textit{strongly persistent} clutter if $I(\mathcal{C})$  has the strong persistence property. 
     }
 \end{definition}


 \begin{definition}   (\cite[Definition  9]{RT})  
 \em{
Given a clutter $\mathcal{C}=(V(\mathcal{C}), E(\mathcal{C}))$ and $x\notin V(\mathcal{C}),$ we define the \textit{cone} over $\mathcal{C}$,  
denoted by $\mathcal{C}x$, the clutter with  vertex set $V(\mathcal{C}x)=V(\mathcal{C})\cup \{x\}$ and edge set 
 $E(\mathcal{C}x)=\{f \cup \{x\} : f\in E(\mathcal{C})\}$. 
 }
 \end{definition}


\begin{notation}
Let  $M\in K[x_1, \ldots, x_n]$ be  a square-free monomial. Then for a monomial   $u=x_1^{\alpha _1}\dotsm x_n^{\alpha_n}\in K[x_1,\ldots, x_n]$, 
 we   set $\deg_M u:=\sum_{x_i \mid M}\alpha_i$.
\end{notation}


The next proposition will be used in the proof of Proposition \ref{strongly persistent.Pro.1}, and we state it here for ease of reference.
\begin{proposition} (\cite[Proposition 5]{RT})          \label{RT. Proposition 5} 
$\mathcal{C}$ has the strong persistence property if and only if $\mathcal{C}x$  has the strong persistence property.
\end{proposition}


To formulate Theorem \ref{main.Section3}, as the main result of this section, we will require Propositions \ref{strongly persistent.Pro.1} and 
\ref{strongly persistent.Pro.2}. 

\begin{proposition}\label{strongly persistent.Pro.1} 
Suppose that  $\mathcal{C}$  is  a strongly persistent clutter, $x\notin V(\mathcal{C})$, and $\mathcal{C}x$ is the cone over $\mathcal{C}$.
 Then $\mathcal{C}'$ is  a  strongly persistent clutter with  $V(\mathcal{C}')=V(\mathcal{C}) \cup \{x\}$ and 
  $E(\mathcal{C}')= \{V(\mathcal{C})\} \cup E(\mathcal{C}x)$.
\end{proposition}

\begin{proof}
Let  $I:=I(\mathcal{C}')$, $J:=I(\mathcal{C}x)$, $\mathcal{G}(J)=\{f_1,\ldots,f_r\}$, and  $f_0=\prod_{z\in V(\mathcal{C})}z$. 
It follows from $E(\mathcal{C}')= \{V(\mathcal{C})\} \cup E(\mathcal{C}x)$ that  $\mathcal{G}(I)=\{f_0, f_1,\ldots,f_r\}$.  
Fix $k\geq 1$  and take an arbitrary monomial   $u\in (I^{k+1}:I)$. Then, for each   $0\leq i \leq r$,   we can write 
 $uf_i=f_0^{\alpha_{i,0}}f_1^{\alpha_{i,1}}\cdots f_r^{\alpha_{i,r}}M_i$  with $\sum_{j=0}^r\alpha_{i,j}=k+1$ and for some monomial $M_i$.  Now, one may consider the  following cases:

\bigskip
\textbf{Case 1.}   $\alpha_{i,0}=0$  for all  $1 \leq i \leq r$. Then $u\in (J^{k+1}:J)$. Since $\mathcal{C}$  is  a strongly persistent clutter, this implies that 
$I(\mathcal{C})$ has the strong persistence property, and  by virtue of  Proposition \ref{RT. Proposition 5}, we can deduce that $J$ has the strong persistence property.  Thus, $u\in J^k$. Due to  $ J^k\subseteq I^k$, we get  $u\in I^k$. 

\bigskip
\textbf{Case 2.} $\alpha_{i,0}>0$ for some $1 \leq i \leq r$. If $x \mid M_i$, then   $f_i\mid M_if_0$ due to $f_0=\prod_{z\in V(\mathcal{C})}z$. We therefore get  
  $u \in I^k$. If $ x\nmid M_i$, then $\deg_xu \leq k-1$, and so $\deg_x{(uf_0)} \leq  k-1$ since  $x\nmid f_0$. 
 Because $uf_0=f_0^{\alpha_{0,0}}f_1^{\alpha_{0,1}}\cdots f_r^{\alpha_{0,r}}M_0$ and $x\mid f_j$ for all $1\leq j \leq r$, we must have $ \alpha_{0,0}\geq 1.$ This   gives rise to  $u\in I^k$.
 
 Consequently, we obtain   $(I^{k+1}:I)=I^k$, and so  $\mathcal{C}'$ is  a  strongly persistent clutter, as claimed.   
\end{proof}


\begin{proposition}  \label{strongly persistent.Pro.2} 
Let $X$ be a finite set and $x,y \notin X$ with  $x\neq y$. Let  $\mathcal{C}$ be a clutter on $X\cup \{x,y\}$ such that $x\in e$ or $y\in e$,
 but $\{x,y\}\nsubseteq e$ for any  $e\in E(\mathcal{C})$.  Then the following statements hold:

\begin{itemize}
\item[(i)]  $\mathcal{C}'$ is strongly persistent, where   $E(\mathcal{C}')=E(\mathcal{C})\cup \{\{x,y\}\}$.
\item[(ii)] $\mathcal{C}'$ is strongly persistent, where   $E(\mathcal{C}')=E(\mathcal{C})\cup \{X,\{x,y\}\}$.
\end{itemize}
\end{proposition}

\begin{proof}
 Suppose that $f_0=\prod_{z\in X}z$, $f_1=xy$, $\mathcal{G}(I(\mathcal{C}))=\{f_2,\ldots,f_r\}$, and  $I:=I(\mathcal{C}')$. Fix $k\geq 1$ and 
 choose an arbitrary monomial $u\in (I^{k+1}:I)$. 

(i)  Since  $E(\mathcal{C}')=E(\mathcal{C})\cup \{\{x,y\}\}$, we have  $\mathcal{G}(I)=\{f_1, f_2,\ldots,f_r\}$. We thus can write  $uf_1=f_1^{\alpha_{1,1}}\cdots f_r^{\alpha_{1,r}}M_1$ with $\sum_{j=1}^r\alpha_{1,j}=k+1$ and some monomial $M_1$. If $x \mid M_1$ or $y \mid M_1$ or $\alpha_{1,1}>0$, then we obtain   $u\in I^k$. Therefore, we assume that  $x\nmid M_1$ , $y \nmid M_1$, and  $\alpha_{1,1}=0$.
  This implies that  $\deg_{xy} u\leq k-1$, and so   $\deg_{xy} (uf_2)\leq k$. 
  Hence, $uf_2\notin I^{k+1}$ since $x \mid f$ or $ y\mid f$ for any $f\in \mathcal{G}(I)$, which  contradicts our  assumption 
   $u\in(I^{k+1}:I)$.

(ii)  Due to  $E(\mathcal{C}')=E(\mathcal{C})\cup \{X,\{x,y\}\}$, we get $\mathcal{G}(I)=\{f_0, f_1, f_2,\ldots,f_r\}$.
  Then, for each   $0\leq i \leq r$,   we can write  $uf_i=f_0^{\alpha_{i,0}}f_1^{\alpha_{i,1}}\cdots f_r^{\alpha_{i,r}}M_i$
    with $\sum_{j=0}^r\alpha_{i,j}=k+1$ and for some monomial $M_i$.  
 If $x \mid M_1$ or $y \mid M_1$ or $\alpha_{1,1}>0$, then we get  $u\in I^k$. Accordingly, we assume that  $x\nmid M_1$, $y \nmid M_1$, and  $\alpha_{1,1}=0$. 
 This yields that  $\deg_{xy} u\leq k-1$, and so  $\deg_{xy} (uf_0)\leq k-1$.  
 Thanks to $uf_0=f_0^{\alpha_{0,0}}f_1^{\alpha_{0,1}}\cdots f_r^{\alpha_{0,r}}M_0$, we must have  $\alpha_{0,0}\geq 1$ due to  $x \mid f_i$ or $ y\mid f_i$ for each $1\leq i \leq r$.  This leads to  $u\in I^k$. 
\end{proof}

The following proposition is essential for us to show Lemma  \ref{strongly persistent.Lem.2}. 

\begin{proposition}\label{strongly persistent.Pro.3} 
Let  $\mathcal{C}$ be  a strongly persistent clutter on a finite set $X$ and $x,y \notin X$ with $x\neq y$. Then $\mathcal{C}'$ is strongly persistent, where 
 $V(\mathcal{C}') =X\cup\{x,y\} $ and $E(\mathcal{C}') =\{e\cup\{x\},e\cup\{y\} : e\in E(\mathcal{C})\}.$
\end{proposition}

\begin{proof}
  Set  $I:=I(\mathcal{C})$ and  $J:=I(\mathcal{C}')$. Let  $\mathcal{G}(I)=\{e_1,\ldots,e_r\}$. Define  $f_i:=e_ix$, and $g_i:=e_iy$ for all $i=1, \ldots, r$. 
This implies that    $\mathcal{G}(J)=\{f_1, \ldots, f_r, g_1, \ldots, g_r\}$. 
 Fix $k\geq 1$ and  take an arbitrary monomial  $u\in (J^{k+1}:J)$. Then we can write $u=u'x^{\alpha} y^{\beta}$ such that 
  $\mathrm{gcd}(u',x^\alpha y^\beta)=1.$ 
 Let  $e_i\in \mathcal{G}(I)$, where $1\leq i \leq r$. Then $ue_ix\in J^{k+1}$, and so we have  
 \begin{equation}
    ue_ix= f_1^{\alpha_1}\cdots f_r^{\alpha_r}g_1^{\beta_1}\cdots g_r^{\beta_r} M=u'x^{\alpha} y^{\beta}e_ix,  \label{eq.4.2}
     \end{equation}
     where $\sum_{i=1}^r \alpha_i+ \sum_{i=1}^r \beta_i=k+1$ and $M$ is a monomial. Also, one can easily
      check that $\deg_{xy}(ue_ix)=\alpha +\beta+1=\deg_{xy} M+k+1\geq k+1$.  This yields that $\alpha + \beta \geq k$. 
           Dividing (\ref{eq.4.2}) by $x^\alpha y^\beta x$, we get    $u'e_i= e_1^{\alpha_1+\beta_1}\cdots e_r^{\alpha_r+\beta_r} M'.$ 
      Because $\sum_{i=1}^r \alpha_i+ \sum_{i=1}^r \beta_i=k+1$ and $e_i$ is arbitrary, this gives that  $u'\in ( I^{k+1}:I)$.
      Since $I$ has the strong persistence property, we obtain   $u'\in I^k$. Thus, we can write  $u'=L e_1^{\gamma_1}\cdots e_r^{\gamma_r}$  
      with $ \gamma_1+\cdots +\gamma_r=k$ and $L$ is a monomial. Due to  $\alpha+\beta \geq k$,  we get the following equalities
 $$u=u'x^\alpha y^\beta =Lx^\alpha y^\beta e_1^{\gamma_1}\cdots e_r^{\gamma_r}=Lx^{\alpha-k_1}y^{\beta -k_2} h_1 \cdots h_k\in J^k,$$ where $k_1+k_2=k$ and, for each $1\leq i\leq k$, we have $h_i= e_jx$ or $h_i=e_jy$ for some $1\leq j\leq r$. This means that $\mathcal{C}'$ is strongly persistent, as required. 
\end{proof}


To demonstrate Proposition \ref{strongly persistent.Pro.4}, we will utilize Lemmas \ref{strongly persistent.Lem.2},  \ref{strongly persistent.Lem.3}, 
and \ref{Lem.Unique}.

\begin{lemma}\label{strongly persistent.Lem.2}
 Let $ X=\{x,y,x_1,x_2,x_3\} $ be a finite set of cardinality $5$, and  $\mathcal{C}$ be a clutter on $X$ such that if $e\in  E(\mathcal{C})$, then 
  $|e|=3$, $\{x,y\}\nsubseteq e$, and $\{x_1,x_2,x_3\}\notin E(\mathcal{C})$.  Then $\mathcal{C}$ is strongly persistent.
\end{lemma}

\begin{proof} 
We put   $I:=I(\mathcal{C})$. It is not hard to check that $E(\mathcal{C})$ is a subset of 
$$\{\{x, x_1 , x_2\}, \{x, x_1 , x_3\}, \{x, x_2 , x_3\}, \{y, x_1 , x_2\}, \{y, x_1 , x_3\}, \{y, x_2 , x_3\}\}.$$
This implies that $x\in e$ or $y\in e$ (not both of them) for each $e\in E(\mathcal{C})$. 
 If $1\leq |E(\mathcal{C})| \leq 2$, then in view of the proof of \cite[Theorem 5.10]{RNA}, we deduce that $I$ has the strong persistence property. 
If  $|E(\mathcal{C})|= 3$, then we can conclude from  Proposition \ref{RT-Pro-8}  that $I$ has the strong persistence property. 
We thus  assume $4\leq |E(\mathcal{C})| \leq 6$. We consider $\mathcal{H}$ and 
 $\mathcal{H}'$ the simple graphs on $X_1\subseteq \{x_1, x_2,x_3\}$ whose edge sets are $E(\mathcal{H})=\{e\cap X_1 \mid e\in E(\mathcal{C}),x\in e\}$ and $E(\mathcal{H}')=\{e\cap X_1 \mid e\in E(\mathcal{C}), y \in e\}$. If  $|E(\mathcal{C})|= 6$, then  $E(\mathcal{H})=E(\mathcal{H}')$, and so 
  one can derive from Proposition \ref{strongly persistent.Pro.3}   that $I$ has the strong persistence property. Hence, the remaining cases are as follows: 

\bigskip 
{\bf Case (I):} $|\mathcal{G}(I(\mathcal{H}))|=3$ and  $|\mathcal{G}(I(\mathcal{H}'))|=1$. Without loss of generality, we can assume 
 $\mathcal{G}(I(\mathcal{H}))=\{f_1=x_1x_2, f_2=x_2x_3, f_3=x_1x_3\}$ and $\mathcal{G}(I(\mathcal{H}'))=\{f_1=x_1x_2\}$. 
 Set  $J:=(f_1, f_2, f_3)$ and pick  an arbitrary   monomial    $m\in(I^{k+1}:I)$, where $k\geq 1$. We can write  $m=m'x^{\alpha} y^{\beta}$ with 
  $\mathrm{gcd}(m',x^\alpha y^\beta)=1.$ Then we have the following  
 \begin{equation}\label{13}
     mf_1y=\ell_1 (f_1x)^{\beta_{11}}(f_2x)^{\beta_{12}}(f_3x)^{\beta_{13}}(f_1y)^{\beta_{14}},
 \end{equation}
 and 
\begin{equation}\label{14}
  mf_ix=\ell_{i2} (f_1x)^{\alpha_{i1}}(f_2x)^{\alpha_{i2}}(f_3x)^{\alpha_{i3}}(f_1y)^{\alpha_{i4} } \ \text{for} \ i =1,2,3,
\end{equation} 
where $\alpha_{i1}+\cdots+\alpha_{i4}=k+1$ for $i=1,2,3,$  $\beta_{11}+\cdots+\beta_{14}=k+1$, and $\ell_1$ and $\ell_{i2}$ (for $i=1,2,3$) are monomials. 
 If $\beta_{14}>0$, then we have $m=\ell_1 (f_1x)^{\beta_{11}}(f_2x)^{\beta_{12}}(f_3x)^{\beta_{13}}(f_1y)^{\beta_{14} -1}\in I^k.$ 
 Thus,  let  $\beta_{14}=0$.  From (\ref{13}), we obtain   $\alpha\geq \beta_{11}+\beta_{12}+\beta_{13}=k+1$. Dividing (\ref{14}) by $x^{\alpha+1} y^\beta$ yields that 
   $$m'f_i=\ell_i' f_1^{\alpha_{i1}+\alpha_{i4}}f_2^{\alpha_{i2}}f_3^{\alpha_{i3}} \   \text{for}   \  i=1,2,3.$$
This gives that  $ m'\in (J^{k+1}:J)$. Since $J$ has  the strong persistence property, we can deduce that  $m'\in J^k$. Thus,
$m'=\ell'' f_1^{\gamma_1}f_2^{\gamma_2} f_3^{\gamma_3}$  with $ \gamma_1+\gamma_2 +\gamma_3=k$ and $\ell''$ is a monomial. 
Accordingly, we get  $$ m=m'x^\alpha y^\beta =\ell''x^{\alpha-k} y^\beta (xf_1)^{\gamma_1}(xf_2)^{\gamma_2} (xf_3)^{\gamma_3} \in I^k.$$ 

\bigskip
{\bf Case (II):}    $|\mathcal{G}(I(\mathcal{H}))|=3$ and  $|\mathcal{G}(I(\mathcal{H}'))|=2$. Without loss of generality, we can assume 
 $\mathcal{G}(I(\mathcal{H}))=\{f_1=x_1x_2,f_2=x_2x_3,f_3=x_1x_3\}$ and  $\mathcal{G}(I(\mathcal{H}'))=\{f_1=x_1x_2,f_3=x_1x_3\}$. 
 Let $m\in(I^{k+1}:I)$ be a monomial and $k\geq 1$. Suppose $m=m'x^{\alpha} y^{\beta}$ with $\mathrm{gcd}(m',x^\alpha y^\beta)=1.$ 
 We thus get  
  \begin{equation} \label{16}
    mf_jy=\ell_{j1} (f_1x)^{\beta_{j1}}(f_2x)^{\beta_{j2}}(f_3x)^{\beta_{j3}}(f_1y)^{\beta_{j4}}(f_3y)^{\beta_{j5}} \ \text{for} \ j=1,3,
\end{equation}
and 
\begin{equation} \label{17}
 mf_ix=\ell_{i2} (f_1x)^{\alpha_{i1}}(f_2x)^{\alpha_{i2}}(f_3x)^{\alpha_{i3}}(f_1y)^{\alpha_{i4}}(f_3y)^{\alpha_{i5}} \ \text{for}  \ i=1,2,3, 
\end{equation}
where $\alpha_{i1}+\cdots+\alpha_{i5}=k+1$ for $i=1,2,3$, $\beta_{j1}+\cdots+\beta_{j5}=k+1$ for $ j=1,3$, and $\ell_{j1}$ (for $j=1,3$)  
and   $\ell_{i2}$ (for $j=1,2,3$)  are monomials. Here, we consider the following subcases:
 \begin{enumerate}
 \item[$\bullet$]    $\beta_{14}>0$ or $\beta_{35}>0$. Then  (\ref{16}) implies that $m\in  I^k$.
   \item[$\bullet$]    $\beta_{11}>0$  and $\beta_{15}>0$. Due to $(f_1x)(f_3y)=(f_1y)(f_3x)$, we can deduce from  (\ref{16})  that $m\in I^k$. 
  \item[$\bullet$]  $\beta_{33}>0$  and $\beta_{34}>0$. It follows from $(f_3x)(f_1y)=(f_3y)(f_1x)$ and (\ref{16})  that $m\in I^k$.  
 \end{enumerate}
     Now, we suppose  that none of the above subcases  are   satisfied.  Then we obtain the following
   \begin{itemize}
   \item[$\bullet$]  $\beta_{14}=0$ and  $\beta_{35}=0$. 
       \item[$\bullet$]   $\beta_{11}=0$  or  $\beta_{15}=0$.
   \item[$\bullet$]    $\beta_{33}=0$  or $\beta_{34}=0$.
      \end{itemize}
Let  $\beta_{11}=0$ and $\beta_{15}=0$. As  $\beta_{14}=0$,   we get $mf_1y=\ell_{11} (f_2x)^{\beta_{12}}(f_3x)^{\beta_{13}},$
 where   $\beta_{12} + \beta_{13}=k+1$. In particular, we must have   $y \mid \ell_{11}$ and $\alpha \geq k+1$. 
 Dividing (\ref{17}) by $x^{\alpha+1} y^\beta$ gives   that 
 $$m'f_i=\ell'_{i2} f_1^{\alpha_{i1}+\alpha_{i4}} f_2^{\alpha_{i2}} f_3^{\alpha_{i3}+\alpha_{i5}} \ \text{for}  \ i=1,2,3.$$
 We thus obtain   $ m'\in (J^{k+1}:J)$. Thanks to  $J$ has  the strong persistence property, one has   $m'\in J^k$, and so  
$m'=\ell'' f_1^{\gamma_1}f_2^{\gamma_2} f_3^{\gamma_3}$  with $ \gamma_1+\gamma_2 +\gamma_3=k$ and $\ell''$ is a monomial. 
Consequently, we can deduce that   $$ m=m'x^\alpha y^\beta =\ell''x^{\alpha-k} y^\beta (xf_1)^{\gamma_1}(xf_2)^{\gamma_2} (xf_3)^{\gamma_3} \in I^k.$$ 
      A similar argument can be repeated for the case in which  $\beta_{33}=0$ and $\beta_{34}=0$. 
      
     Let  $\beta_{11}>0$ and $\beta_{15}=0$. Since $\beta_{14}=0$,  it follows from (\ref{16}) that 
$$mf_1y=\ell_{11}(f_1x)^{\beta_{11}}(f_2x)^{\beta_{12}}(f_3x)^{\beta_{13}}.$$ 
Then $y\mid \ell_{11}$.  Since $\beta_{11}-1+\beta_{12}+\beta_{13}=k$, we can derive that 
  $$m=\frac{\ell_{11}}{y}x(f_1x)^{\beta_{11}-1}(f_2x)^{\beta_{12}}(f_3x)^{\beta_{13}} \in I^k.$$ Therefore, from now on, we assume that 
  $\beta_{11}=0$ and   $\beta_{15}>0$. 
A similar argument can be repeated for the case in which  $\beta_{33}>0$ and $\beta_{34}=0$. Henceforth, we  let  $\beta_{33}=0$ and $\beta_{34}>0$. 

   We can therefore  reduce equations (\ref{16})   to the following equalities
\begin{equation} \label{18}
    m(x_1x_2y)=\ell_{11} (x_2x_3x)^{\beta_{12}}(x_1x_3x)^{\beta_{13}}(x_1x_3y)^{\beta_{15}} 
\end{equation}
and 
\begin{equation} \label{19}
    m(x_1x_3y)=\ell_{31}  (x_1x_2x)^{\beta_{31}}(x_2x_3x)^{\beta_{32}}(x_1x_2y)^{\beta_{34}}
\end{equation}
   with  $\beta_{12}+\beta_{13}+\beta_{15}=k+1$, where $\beta_{15}>0$, and $\beta_{31}+\beta_{32}+\beta_{34}=k+1$, where  $\beta_{34}>0$.
   According to equations (\ref{18}) and (\ref{19}), we get the following inequalities
       \begin{enumerate}
    \item[(i)] $\deg_{x_1}m \geq\max\{\beta_{15}-1,\beta_{34}-1\}$
    \item[(ii)] $\deg_{x_2}m  \geq k+1$ and $\deg_{x_3}m\geq k+1$. 
    \item[(iii)] $\alpha=\deg_{x}m \geq\beta_{12}+\beta_{13}$.
    \item[(iv)] $\beta=\deg_{y}m \geq\beta_{15}-1$.
\end{enumerate}
From (i)-(iv), one can conclude that $x_1^{\beta_{15}-1}x_2^{k+1}x_3^{k+1}x^{\alpha}y^{\beta}\mid m.$ 
   Put $L:=x^{\alpha-\beta_{12}-\beta_{13}}y^{\beta-\beta_{15}+1}x_2^{k+1-\beta_{12}-\beta_{13}}x_3^{k+2-\beta_{12}-\beta_{13}-\beta_{15}}$. 
   On account of 
   $$x_1^{\beta_{15}-1}x_2^{k+1}x_3^{k+1}x^{\alpha}y^{\beta}=L(x_1x_3y)^{\beta_{15}-1}(x_2x_3x)^{\beta_{12}+\beta_{13}}\in I^k,$$
   this implies that $m\in I^k$, as required.

\bigskip
{\bf Case (III):}   $|\mathcal{G}(I(\mathcal{H}))|=|\mathcal{G}(I(\mathcal{H}'))|=2$. Without loss of generality, we can assume 
$\mathcal{G}(I(\mathcal{H}))=\{x_1x_2, x_2x_3\}$ and  $\mathcal{G}(I(\mathcal{H}'))=\{x_1x_2,x_1x_3\}$.
Then  the proof is similar to Case (II).  This finishes the proof. 
\end{proof}


\begin{lemma}\label{strongly persistent.Lem.3}
 Let $ X=\{x,y,x_1,x_2,x_3\} $ be a finite set of cardinality $5$, and  $\mathcal{C}$ be a clutter on $X$ such that if $e\in  E(\mathcal{C})$, then 
  $|e|=3$, $\{x,y\}\nsubseteq e$, and $\{x_1,x_2,x_3\}\in E(\mathcal{C})$.  Then $\mathcal{C}$ is strongly persistent.
\end{lemma}

\begin{proof} 
 One can easily see that  
\begin{align*}
E(\mathcal{C}) \setminus   \{x_1, x_2, x_3\}   \subseteq \{&\{x, x_1 , x_2\}, \{x, x_1 , x_3\}, \{x, x_2 , x_3\}, \{y, x_1 , x_2\},\\
& \{y, x_1 , x_3\}, \{y, x_2 , x_3\}\}.
\end{align*}
Hence, $x\in e$ or $y\in e$ (not both of them) for each $e\in E(\mathcal{C})\setminus \{x_1, x_2, x_3\}$. 
 Put $I:=I(\mathcal{C})$, $f_0=x_1x_2x_3$, $\mathcal{G}(I)=\{f_0, f_1,\ldots, f_t\}$, and  $J:=(f_1, \ldots, f_t)$. 
 Also, for convenience of notation, we set $\{r,s,p\}=\{1,2,3\}$, in particular, we have $f_0=x_rx_sx_p$. We first  note that if $M\in I^k$ is a monomial, then $\deg_{x_1x_2x_3}M\geq 2k$  
 and $\deg_{x_1x_2x_3}{(x_sx_p)}\geq k$. Fix $k\geq 1 $ and take a monomial  $m\in(I^{k+1}:I)$. We can assume that 
 $m=m'x^{\alpha} y^{\beta}$ with $\mathrm{gcd}(m',x^\alpha y^\beta)=1.$ Then we have
 \begin{equation}  \label{15}
   mf_i=\ell_i f_0^{\beta_{i0}}  f_1^{\beta_{i1}}\cdots f_t^{\beta_{it}}  \ \text{for} \ i=0,1,\ldots, t,
\end{equation}
where $\beta_{i0}+ \beta_{i1}+ \cdots+\beta_{it}=k+1$ for each $i=0,1, \ldots,t$ and some monomial $\ell_i$. Here, we consider the following subcases:
\begin{enumerate}
 \item[$\bullet$]   $\beta_{00}>0$ or $f_0\mid \ell_0$.  Then  (\ref{15}) implies that $m\in  I^k$.
    \item[$\bullet$]  $f_0\mid \ell_0f_{j}$ for some $1\leq j \leq t$. According to (\ref{15}), we can deduce that $m\in  I^k$.
  \end{enumerate}
 Suppose, on the contrary,   that none of the above subcases  are   satisfied, and seek a contradiction.  Then we have  the following
   \begin{itemize}
   \item[$\bullet$]  $\beta_{00}=0$ and $f_0\nmid \ell_0$. 
      \item[$\bullet$]  $f_0\nmid \ell_0f_{j}$ for each $1\leq j \leq t$.  
   \end{itemize}
As $\beta_{00}=0$, we can conclude from (\ref{15}) that  $mf_0=\ell_0 f_1^{\beta_{01}}\cdots f_t^{\beta_{0t}}$ such that  $x \mid f_i $  or $y \mid f_i$ for  
each $ 1\leq i \leq t$. Because $f_0\nmid \ell_0$, we will  have three cases for $\ell_0$ as follows: 

\bigskip 
{\bf Case (I):} $\ell_0=x_r^{\alpha''}x_s^{\beta''}x^{\alpha'}y^{\beta'}$  with $\alpha''>0$ and $\beta''>0$. By virtue of $x_r\mid \ell_0$, $x_s\mid \ell_0$, 
and $f_0\nmid \ell_0f_{j}$ for each $1\leq j \leq t$, we must have  $x_p \nmid f_{j}$ for each $1\leq j \leq t$ when $\beta_{0j}>0$.    
 This  implies that  $f_{j}=xx_rx_s$ or $f_{j}=yx_rx_s$  when $\beta_{0j}>0$.  Due to $x_p\nmid \ell_0$, we get  $x_p\nmid \ell_0 f_1^{\beta_{01}}\cdots f_t^{\beta_{0t}}$. On the other hand, since $f_0=x_1x_2x_3$ and $\{1,2,3\}=\{r,s,p\}$, 
 we must have $x_p\mid mf_0$, and so  $x_p\mid \ell_0 f_1^{\beta_{01}}\cdots f_t^{\beta_{0t}}$. This leads to a  contradiction.
 
 \bigskip 
{\bf Case (II):} $\ell_0=x_r^{\alpha''}x^{\alpha'}y^{\beta'}$ with  $\alpha''>0$.  Since  $f_0\nmid \ell_0f_{j}$ for each $1\leq j \leq t$, we may have 
 $f_{j}=xx_rx_s$ or $f_{j}=xx_rx_p$ or $f_{j}=yx_rx_s$ or $f_{j}=yx_rx_p$. Then we get the following, where $\gamma_1+\gamma_2+\gamma_3+\gamma_4=k+1$,
    \begin{equation}\label{10}
    m(x_rx_sx_p)=x_r^{\alpha''}x^{\alpha'}y^{\beta'}(xx_rx_p)^{\gamma_1}(xx_rx_s)^{\gamma_2}(yx_rx_p)^{\gamma_3}(yx_rx_s)^{\gamma_4}. 
\end{equation}
 Because  $x_rx_sx_p\mid x_r^{\alpha''}(xx_rx_p)^{\gamma_1}(xx_rx_s)^{\gamma_2}(yx_rx_p)^{\gamma_3}(yx_rx_s)^{\gamma_4}$, without loss  of generality,  assume $\gamma_3>0$ and $\gamma_4>0$. Dividing  (\ref{10}) by $x_1x_2x_3$, we obtain 
\begin{equation}\label{11}
         m=x_r^{\alpha''+1}x^{\alpha'}y^{\beta'+2}(xx_rx_p)^{\gamma_1}(xx_rx_s)^{\gamma_2}(yx_rx_p)^{\gamma_3-1}(yx_rx_s)^{\gamma_4-1}.
    \end{equation}
  Multiplying  (\ref{11}) by  the  monomial $yx_rx_s\in \mathcal{G}(I)$, we have 
  \begin{align*}
  N=& m(yx_rx_s) \\
  =&x_r^{\alpha''+1}x^{ \alpha'}y^{\beta'+2}(xx_rx_p)^{\gamma_1}(xx_rx_s)^{\gamma_2}(yx_rx_p)^{\gamma_3-1}
  (yx_rx_s)^{\gamma_4} \in I^k.
  \end{align*}  
Since $\deg_{x_sx_p} N=k < k+1 \leq \deg_{x_sx_p} u$, for any monomial $u\in I^{k+1}$, this implies that  $N\notin I^{k+1}$. We thus get 
$m\notin (I^{k+1}:I)$, which is a  contradiction.  

\bigskip 
{\bf Case (III):}  $\ell_0=x^{\alpha'}y^{\beta'}.$  Then we have the following 
   \begin{equation}\label{12}
m(x_rx_sx_p)=x^{\alpha'}y^{\beta'}f_1^{\beta_{01}}\cdots f_t^{\beta_{0t}}.
\end{equation} 
Dividing  (\ref{12})  by $x_1x_2x_3$ and multiplying by $xx_rx_s$, we obtain 
$$M=m(xx_rx_s)=\frac{x^{\alpha'}y^{\beta'}f_1^{\beta_{01}}\cdots f_t^{\beta_{0t}}(xx_rx_s)}{x_1x_2x_3}.$$ 
It is easy to see that   $\deg_{x_1x_2x_3} M=2k+1<2k+2$. This yields that  $M\notin I^{k+1}$, and so $m\notin (I^{k+1}:I)$. This is a contradiction. 

We therefore  conclude that  $m\in I^k$, and the proof is complete.
 \end{proof}


\begin{lemma} \label{Lem.Unique}
 Let $ \mathcal{C}=(V,E)$  be a clutter  with $V(\mathcal{C})=\{x,y,x_1,x_2,x_3\}$  satisfying the following conditions:  
\begin{enumerate}
\item[(1)] $ |e|=3$ for each $e\in E(\mathcal{C})$;
\item[(2)] There exists a unique edge $e \in E(\mathcal{C})$ such that $\{x,y\}\subseteq e$;
\item[(3)] $e\cap\{x,y\}\neq\emptyset$ for each $e\in E(\mathcal{C})$.
\end{enumerate}
Then the following statements hold.
\begin{enumerate}
\item[(i)] $\mathcal{C}$ is strongly persistent.
\item[(ii)] $\mathcal{C}'$ is strongly persistent, where $E(\mathcal{C}')=E(\mathcal{C})  \cup\{\{x_1,x_2,x_3\}\}$.

\end{enumerate}
\end{lemma}
\begin{proof}
(i) We set $I(\mathcal{C})=\{f_1,\ldots,f_r\}$, where $f_1=xyx_1$ is associated to the unique  edge containing $x$ and $y$.
 Let $m\in(I^{k+1}:I)$ be a monomial, where $k\geq 1$. Then, for all $i=1, \ldots, r$,  we can write 
\begin{equation}\label{eq13}
 mf_i=\ell_i f_1^{\alpha_{i,1}}f_2^{\alpha_{i,2}}\cdots  f_r^{\alpha_{i,r}}, 
\end{equation}
where $\sum_{j=1}^r\alpha_{i,j}=k+1$ and $\ell_i$ is a monomial. Let  $J:=(f_2, \ldots, f_r)$.  If  $\alpha_{i,1}=0$ for all  $i=2,\ldots, r$, then
  $m\in (J^{k+1}:J)=J^k$ due to  Lemma \ref{strongly persistent.Lem.2}. 
    Similarly, if $\alpha_{i,i}>0$ for some $1\leq i \leq r$, then  we obtain $m\in I^k$. Hence, we can assume that $\alpha_{i,i}=0$ for  all $i=1,\ldots,r$ and $\alpha_{i_0,1}>0$ for some $2\leq i_0 \leq r$. Thus, we have $\deg_{xy}(mf_{i_0})\geq 2\alpha_{i_0,1}+\alpha_{i_0,2}+\cdots+ \alpha_{i_0,r}\geq k+2$. 
    Now, writing  (\ref{eq13}) for the particular case $i=1$, we get the following 
\begin{equation}\label{eq15}
     m(xyx_1)=\ell_{1}f_2^{\alpha_{1,2}}\cdots  f_r^{\alpha_{1,r}}. 
 \end{equation}
 If $xy \mid \ell_1$ or $xx_1 \mid \ell_1$ or $yx_1 \mid \ell_1$, then we can deduce that $m\in I^k$. 
Therefore, we  can assume that $\ell_1\in\{x^\alpha x_2^\beta x_3^\gamma,y^\alpha x_2^\beta x_3^\gamma, x_1^\alpha x_2^\beta x_3^\gamma\}$. 
Here, one may consider the  following cases:

\bigskip
\textbf{Case 1.}
$\ell_1=x_1^\alpha x_2^\beta x_3^\gamma$. It follows from (\ref{eq15}) that  $\deg_{xy}m=k-1$ and $\deg_{xy}(mf_{i_0})=k$. This leads to 
 a contradiction.

\bigskip
\textbf{Case 2.}
 $\ell_1=x^\alpha x_2^\beta x_3^\gamma$. If $ \alpha=0$, then $ \ell_1=x^\alpha x_2^\beta x_3^\gamma=x_1^\alpha x_2^\beta x_3^\gamma$, and we have Case $1$. Thus, we assume that $\alpha>0$. Since $x_1\nmid \ell_1$, this implies that  $x_1\mid f_j$ for some $2\leq j \leq r$ such that $\alpha_{1,j}>0$. If $y\mid f_j$, then $f_1\mid \ell_1f_j$, and so $m\in I^k$. Hence,  we can assume  $f_j=x_1x_sx$ with $s\in\{2,3\}$. 
 Due to  $y\mid \ell_{1}f_2^{\alpha_{1,2}}\cdots  f_r^{\alpha_{1,r}}$ and $y\nmid \ell_1$, this yields that there exists some $2\leq p\neq j \leq r$ with 
  $\alpha_{1,p}>0$  such that $y\mid f_p$. If $x_1 \mid f_p$, then $f_1 \mid \ell_1f_p$, and so $m\in I^k$. We thus  can assume that 
 $f_p=x_2x_3y$. Hence, we derive that    $\ell_1 f_jf_p=(x^{\alpha-1} x_2^\beta x_3^\gamma x_s)(xyx_1)(xx_2x_3)$. 
  If $xx_2x_3\in I(\mathcal{C})$, then we get  $m\in I^k$. Hence, we let  $xx_2x_3\notin I(\mathcal{C})$.

 Collecting the above assumptions enables us to conclude that for all $c=2, \ldots, r$ with  $\alpha_{1,c}>0$, we have  the following:
 \begin{enumerate}
  \item[$\bullet$] $x_1\mid f_c$ and  $y\nmid f_c$, or
  \item[$\bullet$] $x_1\nmid f_c$ and $y\mid f_c$.
 \end{enumerate}
   We thus  obtain $\deg_{x_1y} m=k-1$ and $\deg_{x_1y}(mf_p)=k$. Considering (\ref{eq13}) for $i=p$, we can deduce that 
   $m(x_2x_3y)=\ell_p f_1^{\alpha_{p,1}}\cdots  f_r^{\alpha_{p,r}}.$
   Recall that  $\alpha_{pp}=0$ and $xx_2x_3\notin I(\mathcal{C})$. If $x_1\nmid  f_j$ and $y\nmid f_j$ for some $1\leq  j \leq r$, then  $f_j=xx_2x_3$, 
   which is a   contradiction. Hence,  $x_1\mid  f_j$ or $y\mid f_j$ for each $1\leq  j \leq r$, and so $\deg_{x_1y}(mf_p)\geq k+1$. This gives   a contradiction. 

\bigskip
\textbf{Case 3.}
$\ell_1=y^\alpha x_2^\beta x_3^\gamma$. The argument  is similar to Case  $2$.

Accordingly, $m\in I^k$, and so the clutter  $\mathcal{C}$ is strongly persistent.

\bigskip
(ii) Set $I:=I(\mathcal{C}')=(f_0=x_1x_2x_3, f_1,\ldots, f_{t-1}, f_t=x_cxy)$, where $x_c\in \{x_1, x_2, x_3\}$.
   Without loss of generality,  assume  that   $f_t=x_1xy$.   Let $k\geq 1$ and take an arbitrary monomial  $m\in(I^{k+1}:I)$. 
  Now, we write $m=m'x^{\alpha} y^{\beta}$ with  $\mathrm{gcd}(m',x^\alpha y^\beta)=1$. Then, for all $i=0,\ldots,t$, we have the following   
\begin{equation}\label{unique.1}
    mf_i=\ell_i f_0^{\beta_{i0}} f_1^{\beta_{i1}}\cdots f_t^{\beta_{it}},
\end{equation}
where $\beta_{i0}+\cdots+\beta_{it}=k+1$ and $\ell_i$ is a monomial. 
   If $\beta_{00}>0$ or $\beta_{tt}>0$, then we can deduce from (\ref{unique.1}) that  $m\in I^k$. Hence, we let  
  $\beta_{00}=0$ and  $\beta_{tt}=0$.
 We  can now deduce from  (\ref{unique.1}) and the assumptions (2) and (3) that 
 \begin{equation}\label{unique.2}
     mf_0=\ell_0 {f_1}^{\beta_{01}}\cdots  {f_{t-1}}^{\beta_{0,t-1}}(x_1xy)^{\beta_{0t}},
  \end{equation}
   and 
  \begin{equation}\label{unique.3}
  mf_t=\ell_t {f_0}^{\beta_{t0}} {f_1}^{\beta_{t1}}\cdots  {f_{t-1}}^{\beta_{t,t-1}},
  \end{equation}
where $ x \mid f_i $  or $ y \mid f_i$ (not both of them) for  all $i=1, \ldots, t-1$.  Assume  that  $\{x_r, x_s, x_p\}=\{x_1,x_2,x_3\}$.  
  It is easy to see that  $|\mathrm{supp}(f) \cap\{x_r, x_s\}|\geq 1$ for each $f\in \mathcal{G}(I) \setminus\{f_t\}$. 
  Hence, (\ref{unique.3}) yields that   $\deg_{x_rx_s}(mf_t)\geq k+1$ and $\deg_{x_1x_2x_3}(mf_t)\geq 2k+2$. 
  Now,  we can write $\ell_0=x_r^{\alpha''}x_s^{\beta''}x_p^{\gamma''} x^{\alpha'} y^{\beta'}$. 
If $x_1x_2\mid \ell_0$ or  $x_1x_3\mid \ell_0$ or $x_2x_3 \mid \ell_0$, then in view of (\ref{unique.2}), we get  $x_1x_2x_3\mid \ell_0f_j$ for some 
$1\leq j \leq t$. This implies that  $m \in I^k$.   Thus, without loss of generality, assume that $\ell_0=x_r^{\alpha''}x^{\alpha'}y^{\beta'}.$ 
Let  $ \alpha''>0$. If there exists some $1\leq j \leq t-1$ with $\beta_{0j}>0$ such that  $f_j=x_sx_px$ or $f_j=x_sx_py$, then
  $f_0\mid \ell_0f_j$, and (\ref{unique.2}) gives  that  $m \in I^k$. 
   We therefore assume that for all $j=1, \ldots, t-1$ with $\beta_{0j}>0$ we have   $f_j\neq x_sx_px$ and  $f_j \neq x_sx_py$. 
 Hence, we obtain  $|\mathrm{supp}(f_j)\cap\{x_r,x_p\}|\leq 1$ for all $j=1, \ldots, t$ with $\beta_{0j}>0$. 
  This leads to $\deg_{x_sx_p}m\leq k-1$, and so $\deg_{x_sx_p}(mf_t)\leq k<k+1$, which is  a contradiction.  This means that this case is impossible.
   If $\alpha''=0$, then  $\deg_{x_1x_2x_3} m=2k+2-\beta_{0,t}-3=2k-1-\beta_{0,t}\leq 2k-1$, and consequently  $\deg_{x_1x_2x_3}(mf_t)\leq 2k<2k+2$. 
   This is a contradiction. Thus, this case is also impossible. Consequently, the clutter  $\mathcal{C}'$ is strongly persistent.
\end{proof}


To complete the proof of Theorem  \ref{main.Section3}, we employ  Propositions \ref{211}, \ref{212}, \ref{213}, and Corollary \ref{strong.persistence.2+4}. 

\begin{proposition}\label{211}
    Let $\mathcal{C}$ be  a clutter with $E(\mathcal{C})=\{e_1, e_2, e_3, \ldots, e_r\}$ such that contains edges $e_1=\{x,y\}$ of cardinality two and  $e_2$ such that $e_1\cap e_2=\emptyset$, and $e_j\cap e_1\neq\emptyset$ for each $3 \leq j \leq r$. Then $\mathcal{C}$ is strongly persistent.
\end{proposition}

\begin{proof}
 Let $I(\mathcal{C})=(f_1, f_2, f_3, \ldots, f_r)$, where $f_i:=\prod_{z\in e_i}z$ for all $i=1, \ldots, r$ with $f_1=xy$ and 
 $\mathrm{supp}(f_1) \cap \mathrm{supp}(f_2)=\emptyset$.   Take an arbitrary monomial  
  $m\in ({I(\mathcal{C})}^{k+1}: I(\mathcal{C}))$, where $k\geq 1$.  Then,   for all  $i=1,\ldots ,r$, 
  we can write  $mf_i=\ell_i f_1^{\alpha_{i1}}\cdots f_r^{\alpha_{ir}}$ with $ \sum_{j=1}^r \alpha_{ij}=k+1$ and $\ell_i$ is a monomial.
    If $\alpha_{11}>0$  or $x\mid\ell_1 $ or $y\mid\ell_1$, then  we can deduce that  $m\in {I(\mathcal{C})}^k$. Hence, 
    we can assume that none of the preceding assumptions are satisfied.    
    Then we have $\alpha_{11}=0$, $x \nmid \ell_1 $, and $y \nmid \ell_1$. We thus get  $\deg_{xy}m\leq k-1$, and so  $\deg_{xy}(mf_2)\leq k-1$ 
due to      $\mathrm{supp}(f_1) \cap \mathrm{supp}(f_2)=\emptyset$.   If $\alpha_{22}=0$, then we must have  $\deg_{xy}(mf_2)\geq k+1$, which is a contradiction. Accordingly, we have  $\alpha_{22}>0$,
     and hence      $m\in {I(\mathcal{C})}^k$. This completes the proof.
\end{proof}


\begin{corollary}\label{strong.persistence.2+4} 
Let $\mathcal{C}$  be a clutter on $X=\{x_1,x_2,x_3,x_4,x_5\}$ such that $2\leq |e|\leq 3$ for each $e\in E(\mathcal{C})$, and $\mathcal{C}$ has only two edges
$e_1, e_2$   with $|e_1|=|e_2|=2$ and  $e_1 \cap e_2=\emptyset$. Then $\mathcal{C}$ is strongly persistent.  
\end{corollary}
\begin{proof} 
One can promptly check that  $|e\cap e_1|=1$ for any $ e\in E(\mathcal{C})\setminus \{e_1,e_2\}$. This implies that the clutter $\mathcal{C}$ satisfies 
the conditions of  Proposition \ref{211}, and so $\mathcal{C}$ is strongly persistent, as desired. 
\end{proof}


\begin{proposition}\label{212}
    Let $X=\{x_1,\ldots,x_r\}$ and $Y=\{y_1,\ldots,y_s\}$  be two disjoint vertex sets and $A \subseteq X$. Let  $\mathcal{C}=(X\cup Y,E)$ be a clutter with  $E(\mathcal{C})=\{\{x_i, y_j\}\mid x_i\in A, y_j\in Y\}\cup\{X\}$. Then $\mathcal{C}$ is strongly persistent.
    \end{proposition}
    
    \begin{proof}
Let $|A|=t$. Without loss of generality, assume that $x_1\in A$. To simplify the notation, we put $I:=I(\mathcal{C})=(f_0, f_1,\ldots,f_{ts})$, where 
    $f_0:=\prod_{z\in X}z$  and  $f_1:=x_1y_1$, and $J:=(f_1,\ldots,f_{ts})$. Fix $k\geq 1$ and let $m\in({I(\mathcal{C})}^{k+1}:I(\mathcal{C}))$  be a monomial. Then we can write
 \begin{equation}\label{1eq}
  mf_i=\ell_i f_0^{\alpha_{i,0}} f_1^{\alpha_{i,1}}\cdots f_{ts}^{\alpha_{i,ts}},
 \end{equation}
  with $ \sum_{j=0}^{ts} \alpha_{i,j}=k+1$   for all $i=0,\ldots ,ts$, and $\ell_i$ is a monomial. 
        If $\alpha_{i,0}=0$ for each $i=1,\ldots, ts$, then we get $m\in (J^{k+1}:J)$. Since $J$ is the edge ideal of a simple graph, this implies that 
 $(J^{k+1}:J)=J^k$, and by virtue of $J^k\subseteq I^k$, we obtain $m\in I^k$.  Thus, without loss of generality, we can assume $\alpha_{1,0}>0$. 
  If  $\alpha_{1,1}>0$ or $x_1\mid \ell_1$ or $y_1\mid\ell_1$, then  we can  conclude that  $m\in I^k.$ So, we additionally assume that $\alpha_{1,1}=0$,  $
  x_1\nmid \ell_1$, and $y_1\nmid\ell_1$. Thus, under this assumptions, (\ref{1eq}) reduces to the following equality
  \begin{equation}\label{2eq}
   m(x_1y_1)=\ell_1 f_0^{\alpha_{1,0}} f_2^{\alpha_{1,2}}\cdots f_{ts}^{\alpha_{1,ts}}.     
     \end{equation}
        
           It follows from (\ref{2eq}) that  $y_1\mid f_j$ for some $2\leq j \leq ts$ with $\alpha_{1,j}>0$. We consider $f_j=x_cy_1$. 
    If $y_i\mid \ell_1$ for some $i\geq 2$, then we can write  
    $$m=\frac{\ell_1}{y_i} \frac{f_0}{x_1} f_0^{\alpha_{1,0}-1} f_2^{\alpha_{1,2}}\cdots \frac{f_j}{x_c}f^{\alpha_{1,j}-1}_j \cdots  f_{ts}^{\alpha_{1,ts}}(x_cy_i). $$
    Since  $x_cy_i\in I$, this leads to $m\in I^k$.     Thus, we  additionally assume that $y_i\nmid \ell_1$ for each $i=1,\ldots, s$. 
     From (\ref{2eq}), we derive that   $\deg_Y m \leq k-1$,  and so $\deg_Y(mf_0)\leq k-1$. 
    Finally, if $ \alpha_{0,0}=0$, then (\ref{1eq}) gives $\deg_Y(mf_0)\geq k+1$, which  is a contradiction. Hence, we must have $\alpha_{0,0}>0$, and 
    therefore $m\in I^k$. This terminates the proof.  
        \end{proof}

\begin{proposition}\label{213}
        Let $X=\{x_1,\ldots,x_r\}$ and $Y=\{y_1,y_2\}$  be  disjoint vertex sets and $A,B \subseteq X$ such that $B\subseteq A$ and $|A\setminus B|=1$. 
 Let  $\mathcal{C}=(X\cup Y,E)$  be a clutter with $E(\mathcal{C})=\{\{x_j,y_1\}\mid x_j\in A\} \cup \{\{x_i,y_2\} \mid x_i\in B\} \cup\{X\}$. 
 Then $\mathcal{C}$ is strongly persistent.
\end{proposition}
 
\begin{proof}
 Let $r=|B|$ and,  without loss of generality,  assume that  $\{x_1\}=A\setminus B$.  
 For convenience of notation, we set    $I:=I(\mathcal{C})=(f_0, f_1, f_2,\ldots, f_{2r+1})$, where $f_0=\prod_{z\in X}z$ and $f_1=x_1y_1$. 
 Choose a monomial  $m\in(I^{k+1}:I)$, where $k\geq 1$.    Hence,  for all $i=0,\ldots, 2r+1$, we can write  
 \begin{equation}\label{4eq}
  mf_i=\ell_i f_0^{\alpha_{i,0}} f_1^{\alpha_{i,1}}\cdots f_{2r+1}^{\alpha_{i,2r+1}},
 \end{equation}
 with $ \sum_{j=0}^{2r+1} \alpha_{i,j}=k+1$ and $\ell_i$ is a monomial.   
            If $\alpha_{ss}>0$ for some $0\leq s \leq 2r+1$ or $x_1\mid \ell_1$ or $y_1 \mid \ell_1$, then  we conclude immediately that  $m\in I^k$. 
 We thus  assume that $\alpha_{ss}=0$ for all $s=0, \ldots, 2r+1$, $x_1\nmid \ell_1$, and $y_1 \nmid \ell_1$. Then $y_1\mid f_j$ for some $2\leq j \leq 2r+1$ 
 with  $\alpha_{1,j}>0$. In particular, we  note that $\deg_{y_1y_2}(mf_0)\geq k+1$. Since $\alpha_{1,1}=0$,  this implies that 
   $f_j=x_sy_1$ with $x_s\neq x_1$.  On account of  $x_1\nmid \ell_1$ and  $x_1\nmid f_i$ for each $2 \leq i \leq 2r+1$, we must have 
   $\alpha_{1,0}>0$. Therefore, we can  rewrite (\ref{4eq}) as follows
                       \begin{equation}\label{5eq}
  m(x_1y_1)=\ell_1 (f_0x_sy_1) f_0^{\alpha_{1,0}-1} f_2^{\alpha_{1,2}}\cdots f_j^{\alpha_{1,j}-1} \cdots f_{2r+1}^{\alpha_{1,2r+1}}. 
 \end{equation} 
 If $y_2\mid \ell_1$, then  $x_sy_2\mid \ell_1f_0x_sy_1$, and  so $m\in I^k$. Finally, if $y_2\nmid \ell_1$, then  we obtain $\deg_{y_1y_2}m \leq k-1$,
  and hence   $\deg_{y_1y_2}(mf_0)\leq k-1$, which is a contradiction. Accordingly, we have $m\in I^k$, and the proof is done. 
\end{proof}


To establish Lemma \ref{strongly persistent.Lem.4}, we first need to recall the following definition. 

\begin{definition} (\cite[Definition 12.6.1]{HH1})
\em{
Let $I\subset K[x_1, \ldots, x_n]$ be a monomial ideal generated in one degree. We say that $I$ is {\it polymatroidal} \index{polymatroidal ideal}    if the following ``exchange condition"  is satisfied: For monomials $u = x_1^{a_1}\cdots x_n^{a_n}$  and $v = x_1^{b_1}\cdots x_n^{b_n}$ belonging to $\mathcal{G}(I)$
and, for each $i$ with $a_i > b_i$, one has $j$ with $a_j < b_j$ such that $x_ju/x_i \in \mathcal{G}(I)$. 
}
\end{definition}


We are ready to state  Lemma \ref{strongly persistent.Lem.4} which will be applied   to prove Proposition \ref{strongly persistent.Pro.4}.

\begin{lemma}\label{strongly persistent.Lem.4} 
Let $\mathcal{C} $ be a clutter on $\{x_1,x_2,x_3,x_4,x_5\}$ such that $|e|=3$ for each $e\in E(\mathcal{C})$. 
If any $x_i,x_j\in V(\mathcal{C})$ with $x_i\neq x_j$ are contained in at least two different edges, then $E(\mathcal{C})$ is the set of bases of a
matroid, and hence $\mathcal{C}$ is strongly persistent. 
\end{lemma}

\begin{proof}
Based on \cite[Proposition 2.4]{HQ}, it is well-known that any polymatroidal ideal satisfies the strong persistence property. Hence it is enough for us to show that 
$I:=I(\mathcal{C})$ is polymatroidal. Let $u, v \in \mathcal{G}(I)$ with $u\neq v$. It is not hard to see that 
 $1\leq |\mathrm{supp}(u) \cap \mathrm{supp}(v)| \leq 2$. Hence, one may consider the  following cases:

\bigskip
\textbf{Case 1.}    $|\mathrm{supp}(u) \cap \mathrm{supp}(v)|=2$. Then $u=x_ix_jx_r$ and $v=x_ix_jx_s$ with $r\neq s$. 
Since $x_su/x_r=v$, we get $x_su/x_r\in \mathcal{G}(I)$.

\bigskip
\textbf{Case 2.}  $|\mathrm{supp}(u) \cap \mathrm{supp}(v)|=1$.  It follows from our hypothesis that there exists some 
$w\in \mathcal{G}(I)$ such that  $|\mathrm{supp}(u) \cap \mathrm{supp}(w)|=2$. Now, we can repeat the argument in Case 1. 

 Consequently, we can deduce that  $E(\mathcal{C})$ is the set of bases of a matroid, and hence $\mathcal{C}$ is strongly persistent. 
 \end{proof}


It should be observed that Proposition \ref{strongly persistent.Pro.4} and  Corollary \ref{strong.persistence.2+4} are crucial for us to complete the proof of Theorem \ref{main.Section3}. 

\begin{proposition}\label{strongly persistent.Pro.4} 
  Let $\mathcal{C} $ be a clutter on $\{x_1,x_2,x_3,x_4,x_5\}$ such that $|e|=3$ for each $e\in E(\mathcal{C})$.  Then $\mathcal{C}$ is strongly persistent. 
\end{proposition}
\begin{proof}
    If $\mathcal{C} $ satisfies the conditions of Lemma \ref{strongly persistent.Lem.4}, then    $\mathcal{C}$ is strongly persistent. 
    In contrary case, there exist  $x,y\in V(\mathcal{C})$  with $x\neq y$ such that are contained in at most one  edge. If both $x $ and $y$ are not contained
     in any edge (i.e., $\{x,y\}\nsubseteq e$ for all $e\in E(\mathcal{C})$), then $\mathcal{C}$ satisfies the condition of Lemma \ref{strongly persistent.Lem.2} or Lemma \ref{strongly persistent.Lem.3}, and so 
     $\mathcal{C}$ is strongly persistent.       Finally, we assume that there exists a unique edge $e_1$ such that contains $x$ and $y$. 
     Then $\mathcal{C}$ satisfies the conditions  of  Lemma \ref{Lem.Unique} part (i) or (ii). Hence, $\mathcal{C}$ is strongly persistent, and the proof is complete. 
     \end{proof}


We are  in a position to express and verify  the   main result of this section in the subsequent theorem.

\begin{theorem} \label{main.Section3}
Every square-free monomial ideal in  $K[x_1,x_2,x_3,x_4,x_5]$ has the strong persistence property. 
\end{theorem}

\begin{proof} 
Let $I\subset R=K[x_1, x_2, x_3, x_4, x_5]$ be a square-free monomial ideal, and  $\mathcal{C}$ be the clutter associated to $I$. 
First note that by the proof of Theorem 5.10 in  \cite{RNA} and Proposition \ref{RT-Pro-8}, we are able to assume that  $|E(\mathcal{C})|>3$.  If $|e|=1$
 for some $e\in E(\mathcal{C})$, then the claim can be deduced from  Lemma \ref{NKRT-Lem-2.2}. 
 We thus  assume that $|e| \geq 2$ for each $e\in  E(\mathcal{C})$. Here, we may consider  the  following cases:

\bigskip
{\bf Case 1.}  $\mathcal{C}$ has some edge of cardinality $4$. Then we can derive from  Theorem  \ref{RT-Th-9}  and 
Proposition \ref{strongly persistent.Pro.1}  that $\mathcal{C}$ is strongly persistent. 

\bigskip 
{\bf Case  2.}  $|e|\leq 3$ for each $e\in E(\mathcal{C})$ and $\mathcal{C}$  has only one edge of cardinality $2$. Then   Proposition \ref{strongly persistent.Pro.2}  implies that  $\mathcal{C}$ is strongly persistent. 

\bigskip 
{\bf Case  3.} $|e|\leq 3$ for each $e\in E(\mathcal{C})$ and $\mathcal{C}$  has exactly two  edges of cardinality $2$, say  $ |e_1|=|e_2|=2$.
  If $e_1\cap e_2\neq\emptyset$, then $\mathcal{C}$ satisfies the condition of Proposition \ref{strongly persistent.Pro.2}. 
  If $e_1\cap e_2=\emptyset$, then $\mathcal{C}$ satisfies the condition of Corollary   \ref{strong.persistence.2+4}.

\bigskip 
{\bf Case 4.}   $|e|\leq 3$ for each $e\in E(\mathcal{C})$ and  $\mathcal{C}$ has at least $3$ edges of cardinality $2$, say $e_1$, $e_2$, and $e_3$. 
 Then it is easy to check that there exist three possibles  for these edges, that is, 
 \begin{itemize} 
 \item $\{e_1, e_2, e_3\}=\{\{x_1, x_2\}, \{x_2,x_3\}, \{x_3, x_4\}\}\; \text{(type I)}.$ 
\item  $\{e_1, e_2, e_3\}=\{\{x_1, x_2\}, \{x_2,x_3\}, \{x_1,x_3\}\}\; \text{(type II)}.$   
 \item $\{e_1, e_2, e_3\}=\{\{x_1, x_2\}, \{x_1,x_3\}, \{x_1,x_4\}\}\; \text{(type III)}.$ 
\end{itemize}
If $ e_i^c=\{x_1,x_2,x_3,x_4,x_5\}\setminus e_i\in E(\mathcal{C})$ for some $1\leq i \leq 3$, 
then $ \mathcal{C}$ satisfies the conditions of  Proposition \ref{strongly persistent.Pro.2}(ii) by  taking $\{x,y\}=e_i$ and $X=e_i^c$. Thus, we can 
assume that $e_i^c\notin E(\mathcal{C})$ for each $i=1,2,3$. In particular, if  $e'\in E(\mathcal{C})$ with $|e'|=3$, then $|e\cap e_i|=1$ for each $i=1,2,3$. 
If  $e_1$, $e_2,$ and $e_3$ are of the type II, then  $\mathcal{C}$ must be  a simple graph since if $|e|=3$, then  $e\cap e_i=e_i$ for some $1\leq i \leq 3$, and so  $e_i\subseteq e$, a contradiction. Hence,  we can additionally  assume that $\mathcal{C}$ does not contain the type II but has the type  III.  We note $x_1\in e$ implies $x_s\notin e$ for $s=2,3,4$. Thus obtain $e=\{x_1,x_5\}$. Similarly,  $ x_1\notin e$ implies $x_5 \notin e$, since $x_5\in e$ yields that  $\{x_1,x_t\},\{x_5,x_r,x_s\} \in E(\mathcal{C})$ for $\{r,s,t\}=\{2,3,4\}$. So $x_1,x_5\notin e$ gives $e=\{x_2,x_3,x_4\}$. We therefore  conclude that the only edge $e$ of cardinality 3 is $e=\{x_2,x_3,x_4\}$, and that  the edges incident at $x_5$ are edges of cardinality $2$. We note that $\{x_1,x_5\}\notin E(\mathcal{C})$ since $\{x_1,x_5\}^c=\{x_3,x_4,x_5\}$. Hence, we reduce to  the following  cases:
   \begin{enumerate}
  \item[(A)] If $\deg(x_5)=3$, we consider  $\{y_1,y_2\}=\{x_1,x_5\}$ in Proposition \ref{212}.
\item[(B)] If $\deg(x_5)=2$, we consider $\{y_1,y_2\}=\{x_1,x_5\}$ in Proposition \ref{213}.
\item[(C)] If $\deg(x_5)=1$,  we consider $f=\{x_i,x_5\}$, where $x_i \in \{x_2, x_3, x_4\}$,   in     Lemma \ref{RT-Lem-2}. 
   
   \end{enumerate}
Case (C) is obtained since  $f\cap e  =\emptyset$ or $f\cap e=\{x_i\}$ or $f\cap e=f$ for each $e\in  E(\mathcal{C})$, and so  $\emptyset \subseteq \{x_5\}\subseteq f$ is a chain.  Accordingly,  we can deduce that  $\mathcal{C}$  is strongly persistent.

Finally, we assume that $\mathcal{C}$ has the types of  neither II nor III but has the type I.  We set $e_1:=\{x_1,x_2\}$, $e_2:=\{x_2,x_3\},$ and   $e_3:=\{x_3,x_4\}$. Then the only  possible for the edges are as follows:
\begin{enumerate}
\item[(i)] If $\{x_1,x_4\}\in E(\mathcal{C})$, then  we can use Proposition \ref{211} when  $f_1=x_2x_3$ and $f_2=x_1x_4$; 
\item[(ii)] If $\{x_1,x_5\}\in E(\mathcal{C})$, then  we can employ   Proposition \ref{211} when $f_1=x_2x_3$ and $f_2=x_1x_5$;
\item[(iii)] If $\{x_1,x_4\},\{x_1,x_5\}\notin E$, then  we can utilize  Proposition  \ref{strongly persistent.Pro.2}(i).
\end{enumerate}

\bigskip 
{\bf Case 5.}  $|e|=3$ for each $e\in E(\mathcal{C})$. One can conclude from Proposition \ref{strongly persistent.Pro.4}
  that  $\mathcal{C}$ is strongly persistent. 
 \end{proof}
 

We terminate this section with the following remark. 

\begin{remark}
\em{
It should be noted that we cannot drop the word ``square-free" in Theorem \ref{main.Section3}, and so it cannot be  extended to any monomial ideal. 
A large class of counterexamples has been presented in \cite[Proposition 3.1]{NKRT} (for $K[x,y]$)
 and  \cite[Example 2.8]{Nasernejad1} (for $K[x,y,z]$).
} 
\end{remark}


\section{A  minimal counterexample to the Conforti-Cornu\'e{j}ols conjecture}

 The main aim of this section is to provide a criterion for a minimal counterexample to the Conforti-Cornu\'e{j}ols conjecture. To accomplish this,
  we require to  give plenty of results in this direction. We start to show Proposition   \ref{Pro.3} and Lemma \ref{Min-existence}, which will be employed in 
   Lemma \ref{Lem.2}.

\begin{proposition} \label{Pro.3}
Let $I\subset R=K[x_1, \ldots, x_n]$ be a square-free  monomial ideal such that there exist a positive integer $\ell$ and some square-free monomial $v\in I^\ell$ with  $(\mathfrak{p}^t:v)=\mathfrak{p}^{t-\ell}$ and $I^{t-\ell}=I^{(t-\ell)}$ for all     $\mathfrak{p}\in \mathrm{Min}(I)$ and  $t>\ell$. Then  
     $(I^t:v)=I^{t-\ell}$    for all $t>\ell$. 
\end{proposition} 

\begin{proof}
Let $t>\ell$. Since $v\in I^\ell$, this implies that $vI^{t-\ell}\subseteq I^t$, and so $I^{t-\ell} \subseteq (I^t:v)$. To complete the proof, we show that $(I^t:v) \subseteq I^{t-\ell}$. Take an arbitrary monomial $u\in (I^t:v)$. Let $\mathfrak{p}\in \mathrm{Min}(I)$. Hence, $I^t\subseteq \mathfrak{p}^t$, and hence   $u\in (\mathfrak{p}^t:v)$. It follows from our assumption that $u\in \mathfrak{p}^{t-\ell}$. Due to $\mathfrak{p}$ is arbitrary, we conclude that 
$u\in \bigcap_{\mathfrak{p}\in \mathrm{Min}(I)}\mathfrak{p}^{t-\ell}$. According to  Definition \ref{Def.4.3.22} and  Fact  \ref{Ex.6.1.25}, we obtain 
$I^{(t-\ell)}= \bigcap_{\mathfrak{p}\in \mathrm{Min}(I)}\mathfrak{p}^{t-\ell}$, and by virtue of  $I^{t-\ell}=I^{(t-\ell)}$, we get 
$u\in I^{t-\ell}$. Therefore,  $(I^t:v) \subseteq I^{t-\ell}$, as required. 
\end{proof}


 \begin{lemma} \label{Min-existence}
Let $I\subset R=K[x_1, \ldots, x_n]$ be a  square-free monomial ideal  and $\mathfrak{p}\in \mathrm{Ass}(R/I^s)\setminus \mathrm{Min}(I)$
  for some  $s> 1$ and  $x_i \in \mathfrak{p}$ with $1\leq i \leq n$.  Then the following statements hold:
  \begin{itemize}
 \item[(i)]    there exists some $\mathfrak{q}\in  \mathrm{Min}(I)$ such that $x_i\in \mathfrak{q}$   and $\mathfrak{q} \subsetneq \mathfrak{p}$.     
  \item[(ii)]  $\mathfrak{p}\setminus x_i \notin \mathrm{Min}(I\setminus x_i)$. 
  \end{itemize}
  \end{lemma}

\begin{proof}

(i)  We first note that if $|\mathrm{Min}(I)|=1$, say $\mathrm{Min}(I)=\{\mathfrak{q}\}$, then $I=\mathfrak{q}$, and so 
$\mathrm{Ass}(R/I^s)=\{\mathfrak{q}\}$, which contradicts our assumption $\mathfrak{p}\in \mathrm{Ass}(R/I^s)\setminus \mathrm{Min}(I)$. 
Hence, we must have $|\mathrm{Min}(I)|\geq 2$.   Let $\mathrm{Min}(I)=\{\mathfrak{q}_1, \ldots, \mathfrak{q}_r\}$, where $r\geq 2$,  and
  $\mathcal{G}(I^s)=\{u_1, \ldots, u_m\}$. Since $\mathfrak{p}\in \mathrm{Ass}(R/I^s)$, 
there exists some monomial   $h\in R$ such that $\mathfrak{p}=(I^s:h)$, and so $\mathfrak{p}=\sum_{i=1}^m(u_i:h)$. Due to $x_i\in \mathfrak{p}$, 
we get $x_i\in (u_\lambda:h)$ for some $1\leq \lambda \leq m$. Hence, $x_ih=u_\lambda w$ for some monomial $w$ in $R$.   Since $h\notin I^s$, we obtain $x_i \nmid w$, and so $x_i \mid u_\lambda$. Because $u_\lambda \in \mathcal{G}(I^s)$, there exist $v_1, \ldots, v_s\in \mathcal{G}(I)$ such that $u_\lambda=v_1 \cdots v_s$,  and hence $x_i \mid v_1 \cdots v_s$. This implies that $x_i \mid v_c$ for some $1\leq c \leq s$. Let $v_c=x_ig$ for some monomial $g$ in $R$. It 
 follows from $v_c\in I$ that $v_c \in \mathfrak{q}$ for all $\mathfrak{q}\in \mathrm{Min}(I)$. On the contrary, assume that $x_i\notin \mathfrak{q}$ for 
 all $\mathfrak{q}\in \mathrm{Min}(I)$. We thus get $g\in \mathfrak{q}$ for all $\mathfrak{q}\in \mathrm{Min}(I)$, and so $g\in I$. This contradicts 
 the minimality of $v_c$. Consequently, there exists some $\mathfrak{q}\in \mathrm{Min}(I)$ such that $x_i\in \mathfrak{q}$.  If $\mathfrak{p}=\mathfrak{m}$, then  the proof is over. Thus, we let $\mathfrak{p}\subsetneq \mathfrak{m}$. Without loss of generality,  assume that $\mathfrak{q}_j\subset \mathfrak{p}$ for all $j=1, \ldots, z$, where $z\geq 1$,   and $\mathfrak{q}_j\not\subset\mathfrak{p}$ for all $j=z+1, \ldots, r$.  
  Suppose, on the contrary, that $x_i\notin \mathfrak{q}_j$ for all $j=1, \ldots, z$.  
 Put $t:=\prod_{x_{\alpha}\in \mathfrak{m}\setminus \mathfrak{p}}x_{\alpha}$. Since  $\mathfrak{p}\subsetneq \mathfrak{m}$, this implies that   $t\neq 1$. 
  We now localize at $t$. It follows immediately  from   \cite[Lemma 9.38]{sharp} that 
   $$\mathrm{Ass}_{R_t}(I^sR_t)=\{\mathfrak{q}R_t : \mathfrak{q} \in \mathrm{Ass}(I^s) \; \text{and} \; \mathfrak{q} \cap T=\emptyset\},$$ 
   where $T=\{t^n : n\in \mathbb{N}\cup \{0\}\}$ is the  multiplicatively closed subset of $R$. 
   In particular, we have $\mathrm{Ass}_{R_t}(IR_t)=\{\mathfrak{q}_1R_t, \ldots, \mathfrak{q}_zR_t\}$ and $\mathfrak{p}R_t \in \mathrm{Ass}_{R_t}(I^sR_t)$. 
 Because  $x_i \in \mathfrak{p}$,  this yields that $x_i\in \mathfrak{p}R_t$, and so $x_i \in \mathrm{supp}(I^sR_t)$. Due to 
 $\mathrm{supp}(IR_t)=\mathrm{supp}(I^sR_t)$, we must have $x_i \in \mathrm{supp}(IR_t)$. On account of  $x_i \notin \bigcup_{j=1}^z \mathfrak{q}_jR_t$, 
 we can deduce that $x_i \notin \mathrm{supp}(IR_t)$. This leads to a contradiction. Accordingly, we must have $x_i \in \mathfrak{q}_j$ for some 
 $1\leq j \leq z$, and the proof is complete.   

(ii) It follows from part  (i) that there exists some $\mathfrak{q}\in  \mathrm{Min}(I)$ such that $x_i\in \mathfrak{q}$ 
  and $\mathfrak{q} \subsetneq \mathfrak{p}$.      In particular, we have $\mathfrak{q}\setminus x_i \subsetneq \mathfrak{p}\setminus x_i$. 
  If   $\mathfrak{q}\setminus x_i \in \mathrm{Min}(I\setminus x_i)$, then $\mathfrak{p}\setminus x_i \notin \mathrm{Min}(I\setminus x_i)$, and the proof is 
  done. Hence, let $\mathfrak{q}\setminus x_i \notin \mathrm{Min}(I\setminus x_i)$. This means that there exists some $\mathfrak{q}'\in \mathrm{Min}(I\setminus x_i)$ such that $\mathfrak{q}' \subsetneq \mathfrak{q}\setminus x_i$, and hence $\mathfrak{q}' \subsetneq  \mathfrak{p} \setminus x_i$. This gives rise to 
     $\mathfrak{p}\setminus x_i \notin  \mathrm{Min}(I\setminus x_i)$, as required. 
\end{proof}


\begin{corollary} \label{Cor.2}
 Let $I\subset R=K[x_1, \ldots, x_n]$ be a  square-free monomial ideal,  $\mathfrak{p}$ a prime monomial ideal in $R$ with  $x_i\in \mathfrak{p}$, where $1\leq i \leq n$,   such that   $\mathfrak{p}\setminus x_i \in \mathrm{Min}(I\setminus x_i)$.  Then $\mathfrak{p}\notin \mathrm{Ass}(R/I^s)\setminus \mathrm{Min}(I)$ 
 for all $s>1$. 
  \end{corollary}


\begin{remark}
\em{
It should be noted that we cannot extend  Lemma \ref{Min-existence}(i) to any monomial ideal. To see a counterexample, consider the monomial ideal 
$I = (y^2, xy)$ in $R=K[x,y]$. Then it is easy to check that $\mathrm{Min}(I)=(y)$ and $\mathfrak{p}=(x,y)\in \mathrm{Ass}(I^2)\setminus \mathrm{Min}(I)$ and $x\in \mathfrak{p}$, while $x\notin (y)$. 
}
\end{remark}


We are ready to prove Lemma \ref{Lem.2}, which will be used in Theorem \ref{Th.1}.

\begin{lemma}\label{Lem.2}
Let $I\subset R=K[x_1, \ldots, x_n]$ be a square-free monomial ideal such that there exist  a positive integer $\ell$ and some square-free monomial $v\in I^\ell$ with $(\mathfrak{p}^t:v)=\mathfrak{p}^{t-\ell}$ for all   $\mathfrak{p}\in \mathrm{Min}(I)$  and  $t>\ell$,  and $I\setminus x_i$ is normally torsion-free for all $x_i\in \mathrm{supp}(v)$.  Then $I$ is normally torsion-free. 
\end{lemma}

\begin{proof}
We assume $I$ is not normally torsion-free and seek a contradiction.  This implies that there exists a positive integer $s$ such that $I^s$ has embedded primes. Suppose that $s$ is minimal with respect to this property.  Due to $I$ is a square-free monomial ideal, we have $\mathrm{Ass}(R/I)=\mathrm{Min}(I)$, and so  $s\geq 2$. Let $\mathfrak{q}$ be an embedded prime ideal of $I^s$.  Since   $I\setminus x_i$ is normally torsion-free    for all $x_i\in \mathrm{supp}(v)$, we have 
$\mathrm{Ass}(R/(I\setminus x_i)^s)=\mathrm{Min}(I\setminus x_i)$   for all $x_i\in \mathrm{supp}(v)$. Fix   $x_i\in \mathrm{supp}(v)$. 
 In what follows, we are going to verify that $\mathfrak{q}\setminus x_i \notin \mathrm{Min}(I\setminus x_i)$. To see this, we consider the following cases: 

  \bigskip
\textbf{Case 1.}   $x_i\notin \mathfrak{q}$. It follows from $\mathfrak{q}\in \mathrm{Ass}(R/I^s) \setminus \mathrm{Min}(I)$ that there exists some 
  $\mathfrak{p}\in  \mathrm{Min}(I)$ such that  $\mathfrak{p} \subsetneq \mathfrak{q}$. Since $x_i \notin \mathfrak{q}$, this yields that 
  $x_i \notin \mathfrak{p}$. In particular, we have $\mathfrak{q}=\mathfrak{q}\setminus x_i$ and $\mathfrak{p}=\mathfrak{p}\setminus x_i$.  
  If   $\mathfrak{p}\setminus x_i \in \mathrm{Min}(I\setminus x_i)$, then $\mathfrak{q}\setminus x_i \notin \mathrm{Min}(I\setminus x_i)$, and our argument is over. We thus assume that  $\mathfrak{p}\setminus x_i \notin \mathrm{Min}(I\setminus x_i)$. This gives  that there exists some 
  $\mathfrak{p}'\in \mathrm{Min}(I\setminus x_i)$ such that $\mathfrak{p}' \subsetneq \mathfrak{p}\setminus x_i$, and so  
  $\mathfrak{p}' \subsetneq  \mathfrak{q} \setminus x_i$. Consequently, one has   $\mathfrak{q}\setminus x_i \notin  \mathrm{Min}(I\setminus x_i)$.  
  
\bigskip
\textbf{Case 2.}  $x_i\in \mathfrak{q}$. In the light of   Lemma  \ref{Min-existence}(ii), we can deduce that  $\mathfrak{q}\setminus x_i \notin
 \mathrm{Min}(I\setminus x_i)$.

\bigskip
We therefore get  $\mathfrak{q}\setminus x_i \notin \mathrm{Ass}(R/(I\setminus x_i)^s)$  for all  $x_i \in \mathrm{supp}(v)$.  In view of Proposition \ref{Pro.1}, we have $\mathfrak{q}\in \mathrm{Ass}(R/(I^s:v))$.    It follows also  from Proposition  \ref{Pro.2}  that $s>\ell$. Thanks to $s$ is minimal, we get $I^{s-\ell}$ has no embedded  primes, and thus  $I^{s-\ell}=I^{(s-\ell)}$. We can  conclude from Proposition \ref{Pro.3} that $(I^s:v)=I^{s-\ell}$. This gives that
   $\mathfrak{q} \in \mathrm{Ass}(R/I^{s-\ell})$, which contradicts our assumption that $I^{s-\ell}$ has no  embedded primes. Accordingly, 
   this gives rise to     $I$ is normally torsion-free, and the proof is done.  
\end{proof}


We are  in a position to establish the first  main result of this section in the following theorem. Particularly, in Example \ref{Exa.NTF.1}, we will show that 
 how one can use Theorem \ref{Th.1} to detect  the normally torsion-freeness of square-free monomial ideals.

\begin{theorem} \label{Th.1}
Let $I\subset R=K[x_1, \ldots, x_n]$ be a square-free monomial ideal such that there exist  a positive integer $\ell$ and some square-free monomial $v\in I^\ell$  such that  $v\in \mathfrak{p}^\ell\setminus \mathfrak{p}^{\ell+1}$ for any $\mathfrak{p}\in \mathrm{Min}(I)$, and $I\setminus x_i$ is normally torsion-free for all  $x_i\in \mathrm{supp}(v)$.  Then $I$ is normally torsion-free. 
\end{theorem}

\begin{proof}
Suppose, on the contrary, that $I$ is not normally torsion-free. Also, assume that $s$ is  minimal such that $I^s$ has embedded prime ideals. 
 By virtue of   $I$ is a square-free monomial ideal, we must have   $s\geq 2$. In what follows, our aim is to show that 
  $(\mathfrak{p}^t:v)=\mathfrak{p}^{t-\ell}$ for all   $\mathfrak{p}\in \mathrm{Min}(I)$  and  $t>\ell$. 
Let  $\mathfrak{p}\in \mathrm{Min}(I)$  and  $t>\ell$.  Since $v \in I^{\ell}$, this gives that $v\in \mathfrak{p}^\ell$, and thus   
$\mathfrak{p}^{t-\ell} \subseteq (\mathfrak{p}^t:v)$.
 To establish the reverse inclusion, consider an arbitrary monomial $u \in  (\mathfrak{p}^t:v)$. Hence, $uv\in \mathfrak{p}^t$.  This implies that 
  there exist  variables  $x_{i_1},  \ldots,  x_{i_t} \in \mathcal{G}(\mathfrak{p})$ and some monomial $f$ in $R$ such that
  $uv=x_{i_1}  \cdots x_{i_t}f$.   Since  $v\in \mathfrak{p}^\ell\setminus \mathfrak{p}^{\ell+1}$ and $v$ is a square-free monomial,
   we  can derive  that $\mathfrak{p}$ contains exactly $\ell$ variables    that divide  $v$. This   leads to  $u\in \mathfrak{p}^{t-\ell}$. 
   Consequently, we get  $(\mathfrak{p}^t:v)=\mathfrak{p}^{t-\ell}$. It follows now from Lemma \ref{Lem.2}  that   $I$ is normally torsion-free, as claimed. 
  \end{proof}


Here, we recall the Conforti-Cornu\'e{j}ols conjecture. In fact, in  1990 \cite{CC}, Michele  Conforti and  G\'e{r}ard   Cornu\'e{j}ols     made up  the following conjecture, and it is still open.

\begin{conjecture} (\cite[Conjecture 1.6]{Cornuejols})  \label{Cornuejols}
A clutter has the MFMC property if and only if it has the packing property.
\end{conjecture}


We now turn our attention to  the  following corollary which is the second main result of this section.

\begin{corollary}\label{Cor.3}
Let $I\subset R=K[x_1, \ldots, x_n]$ be a   minimal counterexample to the Conforti-Cornu\'e{j}ols conjecture. Then  for all positive integer $\ell$ and 
 square-free monomial $v\in I^\ell$,  there exists some   $\mathfrak{p}\in \mathrm{Min}(I)$ such that $v\in  \mathfrak{p}^{\ell+1}$. 
In particular, for all $v \in \mathcal{G}(I)$,  there exists some   $\mathfrak{p}\in \mathrm{Min}(I)$ such that $v\in  \mathfrak{p}^2$. 
\end{corollary}

\begin{proof}
Due to  $I$ is a  minimal counterexample to the Conforti-Cornu\'e{j}ols conjecture, this means that $I$ satisfies the packing property, 
but $I$ is not normally torsion-free, and also  all its proper minors satisfy the packing property, and so are normally torsion-free. 
This yields that  for all positive integer $\ell$ and square-free monomial $v\in I^\ell$,  we can deduce that  $I\setminus x_i$ is normally
 torsion-free for all $x_i\in \mathrm{supp}(v)$. Now, the claim is an immediate consequence of Theorem \ref{Th.1}.  
\end{proof}


\begin{remark} \label{Rem. NTF}
\em{
It has already been stated  in \cite[Corollary 3.10]{HM} that a minimal counterexample to the Conforti-Cornu\'e{j}ols conjecture cannot be unmixed. 
 In comparison with this result, Corollary  \ref{Cor.3} gives another criterion.  To see an example, consider the following square-free monomial ideal 
 $$I=(x_6x_7x_8, x_5x_6x_7, x_1x_2x_7, x_1x_2x_3, x_3x_4x_5x_6, x_2x_3x_4x_5) \subset K[x_1, \ldots, x_8].$$
   Using \textit{Macaulay2} \cite{GS}, we get  
  \begin{align*}
  \mathrm{Min}(I)=\mathrm{Ass}(I)=  \{&(x_6,x_2), (x_7,x_3), (x_6,x_3,x_1), (x_6,x_4,x_1), (x_7,x_4,x_1),  \\
  & (x_6,x_5,x_1), (x_7,x_5,x_1),  (x_8,x_5,x_1), (x_7,x_4,x_2), \\
  &(x_7,x_5,x_2),   (x_8,x_5,x_2)\}.
  \end{align*}
    We deduce that $I$ is not unmixed. Hence, we cannot use \cite[Corollary 3.10]{HM}, and so we cannot judge whether $I$ is a minimal counterexample to the 
  Conforti-Cornu\'e{j}ols conjecture or not, while Corollary  \ref{Cor.3} enables us to argue on it. To do this, let $\ell:=2$ and $v:=x_6x_7x_8 \in I$. It is not hard to check that 
  $v\notin \mathfrak{p}^2$ for all $\mathfrak{p}\in \mathrm{Min}(I)$. This implies that $I$ is  not a minimal counterexample to the Conforti-Cornu\'e{j}ols conjecture. However, we will show in Example \ref{Exam. Linear. C8}  that $I$ is indeed normally torsion-free. 
  }
\end{remark}


In order to demonstrate Proposition \ref{Pro.6}, we require to utilize the next result. 
 
 \begin{proposition} \label{Pro.5}
 Assume that  $Q=(x^{\alpha_1}_{i_1}, \ldots, x^{\alpha_r}_{i_r}) \subset  R=K[x_1, \ldots, x_n]$ is an irreducible primary monomial ideal  with $\alpha_1, \ldots, \alpha_r$ are positive integers. Then $\left(Q^k:_R \prod_{j=1}^sx^{\alpha_j}_{i_j}\right)=Q^{k-s}$ for all $k>s$ and $s\leq r$. 
 \end{proposition}

\begin{proof}
Let $k>s$ and $s\leq r$.  Since  $x^{\alpha_j}_{i_j} \in Q$  for all $j=1, \ldots, s$, this gives that $ \prod_{j=1}^sx^{\alpha_j}_{i_j}\in Q^s$, and so $Q^{k-s} \subseteq \left(Q^k:_R \prod_{j=1}^sx^{\alpha_j}_{i_j}\right)$. To finish the argument,  one has  to   show the opposite inclusion, that is, 
$\left(Q^k:_R \prod_{j=1}^sx^{\alpha_j}_{i_j}\right) \subseteq Q^{k-s}$. To accomplish  this, we proceed by induction on $s$. Let $s=1$.
  Here,   one can conclude  the following equalities
\begin{align*}
 Q^{k}:_R (x^{\alpha_1}_{i_1})& =\left(\displaystyle\sum_{\lambda_1+\cdots+\lambda_r =k} (x^{\alpha_1}_{i_1})^{\lambda_1}\cdots 
  (x^{\alpha_r}_{i_r})^{\lambda_r}\right):_R (x^{\alpha_1}_{i_1})\\
   &= \displaystyle\sum_{\lambda_1+\cdots+\lambda_r =k} (x^{\alpha_1}_{i_1})^{\lambda_1}\cdots 
  (x^{\alpha_r}_{i_r})^{\lambda_r}:_R (x^{\alpha_1}_{i_1}) \\
 &=\displaystyle\sum_{\lambda_1=0,~\lambda_2+\cdots+\lambda_r =k} (x^{\alpha_2}_{i_2})^{\lambda_2}\cdots 
  (x^{\alpha_r}_{i_r})^{\lambda_r}
\\ & 
+ \displaystyle\sum_{\lambda_1\geq 1,~\lambda_1+\cdots+\lambda_r =k} (x^{\alpha_1}_{i_1})^{\lambda_1-1} 
(x^{\alpha_2}_{i_2})^{\lambda_2} \cdots   (x^{\alpha_r}_{i_r})^{\lambda_r}.
   \end{align*}
In the light of        $(x^{\alpha_1}_{i_1})^{\lambda_1-1}\cdots   (x^{\alpha_r}_{i_r})^{\lambda_r} \subseteq Q^{k-1}$ with $\lambda_1\geq 1,~\lambda_1+\cdots+\lambda_r =k$ and $(x^{\alpha_2}_{i_2})^{\lambda_2}\cdots   (x^{\alpha_r}_{i_r})^{\lambda_r}  \subseteq Q^{k-1}$ with 
  $\lambda_2+\cdots+\lambda_r=k$, we get   $\left( Q^{k}:_R (x^{\alpha_1}_{i_1})\right) \subseteq  Q^{k-1}$. Hence, the claim holds for $s=1$.  
  Suppose, inductively, that $s>1$ and that the result has been  shown  for all values less than $s$. The  inductive  hypothesis gives that $\left(Q^k:_R \prod_{j=1}^{s-1}x^{\alpha_j}_{i_j}\right) \subseteq   Q^{k-s+1}$, and so  we have the following 
  \begin{align*}
  \left(Q^k:_R \prod_{j=1}^{s}x^{\alpha_j}_{i_j}\right)&= \left(\left(Q^k:_R \prod_{j=1}^{s-1}x^{\alpha_j}_{i_j}\right):_Rx^{\alpha_s}_{i_s}\right)\\
& \subseteq \left(Q^{k-s+1}:_R x^{\alpha_s}_{i_s}\right)\\
&\subseteq Q^{k-s}.
\end{align*}   
    This completes the inductive step, and hence the assertion  has been shown by induction.    
\end{proof}


\begin{proposition}\label{Pro.6}
Let $I\subset R=K[x_1, \ldots, x_n]$ be a square-free monomial ideal  and 
$\{u_1, \ldots, u_{\beta_1(I)}\}$  be a maximal independent set of minimal generators of $I$  and $\mathfrak{p} \in \mathrm{Min}(I)$.
If  $\mathrm{ht}\mathfrak{p}={\beta_1(I)}$, then  $\left(\mathfrak{p}^{r}:_R\prod_{i=1}^{\beta_1(I)}u_i\right)=\mathfrak{p}^{r-{\beta_1(I)}}$ for all $r>{\beta_1(I)}$.  
\end{proposition}

\begin{proof}
We  put $m:={\beta_1(I)}$. Let   $\mathrm{ht}\mathfrak{p}=m$.   As  $u_i\in\mathfrak{p}$  for all $i=1, \ldots, m$,  
and  $\mathrm{gcd}(u_i,u_j)=1$ for any $1\leq i\neq j\leq m$,  we get     $|\mathrm{supp}(u_i) \cap \mathrm{supp}(\mathfrak{p})|=1$ for all  $i=1, \ldots, m$.  Without loss of generality, assume   $\mathrm{supp}(u_i) \cap \mathrm{supp}(\mathfrak{p})=\{x_i\}$  for all $i=1, \ldots, m$. 
Hence, we  can write $\prod_{i=1}^mu_i=\omega\prod_{i=1}^mx_i$ such that $\mathrm{gcd}(\omega, \prod_{i=1}^mx_i)=1$. 
It follows from   Lemma  \ref{Kaplansky}   that  
$$\left(\mathfrak{p}^{r}:_R\prod_{i=1}^mu_i\right)=\left(\mathfrak{p}^{r}:_R\omega\prod_{i=1}^mx_i\right)=\left(\mathfrak{p}^{r}
:_R\prod_{i=1}^mx_i\right).$$ 
 We can also deduce  from Proposition \ref{Pro.5} that  $\left(\mathfrak{p}^{r}:_R\prod_{i=1}^mx_i\right)=\mathfrak{p}^{r-m}.$
  We thus obtain  $\left(\mathfrak{p}^{r}:_R\prod_{i=1}^mu_i\right)=\mathfrak{p}^{r-m}$, and the proof is done.  
 \end{proof}


\begin{remark}
\em{
It should be observed that the converse of Proposition \ref{Pro.6} ﻿is not always true. To see a counterexample, consider the square-free monomial ideal 
$I=(x_1x_2x_3, x_4x_5, x_6x_7x_8, x_1x_3x_4, x_4x_6x_9)\subset R=K[x_1, \ldots, x_9]$. It is routine to check that 
$\{u_1:=x_1x_2x_3, u_2:=x_4x_5, u_3:=x_6x_7x_8\}$  is  a maximal independent set of minimal generators of $I$. This implies that 
$\beta_1(I)=3$.  Using \textit{Macaulay2} \cite{GS}, we obtain $\mathfrak{p}:=(x_1, x_5, x_8, x_9)\in \mathrm{Min}(I)$. We claim that 
$\left(\mathfrak{p}^{r}:_R\prod_{i=1}^{3}u_i\right)=\mathfrak{p}^{r-3}$ for all $r> 3$. Let $r>3$. 
Plugging Proposition  \ref{Pro.5} into   Lemma  \ref{Kaplansky},   we obtain  the following equalities 
\[
\left(\mathfrak{p}^{r}:_R\prod_{i=1}^{3}u_i\right)=\left(\mathfrak{p}^{r}:_R (x_1x_5x_8)(x_2x_3x_4x_6x_7)\right)
=\left(\mathfrak{p}^{r}:_R   x_1x_5x_8\right)=\mathfrak{p}^{r-3}.
\]
 Because  $\mathrm{ht}\mathfrak{p}=4$ and  $\beta_1(I)=3$, this shows that the converse of Proposition \ref{Pro.6} ﻿is not always true, as  required.    
}
\end{remark}


We can now combine together  Lemma \ref{Lem.2} and Proposition  \ref{Pro.6}  to recover  Theorem 1.1 in \cite{SNQ1} in  the following corollary.
\begin{corollary}\label{Criterion.NTF2}
Let $I\subset R=K[x_1, \ldots, x_n]$ be a square-free monomial ideal  and 
$\{u_1, \ldots, u_{\beta_1(I)}\}$  be a maximal independent set of minimal generators of $I$  such that $\mathrm{ht}\mathfrak{p}={\beta_1(I)}$ 
for each  $\mathfrak{p} \in \mathrm{Min}(I)$, and   $I\setminus x_i$ is normally torsion-free for all  $x_i\in \mathrm{supp}(\prod_{i=1}^{\beta_1(I)}u_i)$. 
 Then $I$ is normally torsion-free. 
\end{corollary}


\section{On the normally torsion-freeness of  linear combinations of two normally torsion-free ideals}

  This section is devoted to giving   a necessary and sufficient condition to determine the normally torsion-freeness  of a linear combination of two  
  normally torsion-free square-free monomial  ideals. To reach this  goal, we begin with the following proposition which will be used in  
  Proposition  \ref{Bipartite}.

  
  \begin{proposition}\label{Intersection}
  Let $I$ and $J$ be two square-free monomial ideals in a polynomial ring $R=K[x_1, \ldots, x_n]$ over a field $K$. Then, for all $k\geq 1$, we have 
  $(I\cap J)^{(k)}=I^{(k)} \cap J^{(k)}$. 
  \end{proposition}
  
  \begin{proof}
  Let $\mathrm{Min}(I)=\{\mathfrak{p}_1, \ldots, \mathfrak{p}_s\}$ and $\mathrm{Min}(J)=\{\mathfrak{q}_1, \ldots, \mathfrak{q}_t\}$. 
   Now, we may consider the following cases: 
   
   \bigskip
\textbf{Case 1.}   $\mathfrak{p}_i \nsubseteq  \mathfrak{q}_j$ and $  \mathfrak{q}_j \nsubseteq \mathfrak{p}_i$ for all $i=1, \ldots, s$ and 
$j=1,\ldots, t$. This implies that $\mathrm{Min}(I\cap J)=\{\mathfrak{p}_1, \ldots, \mathfrak{p}_s, \mathfrak{q}_1, \ldots, \mathfrak{q}_t\}=
\mathrm{Min}(I)\cup \mathrm{Min}(J).$ It follows from    Definition \ref{Def.4.3.22}    and   Fact  \ref{Ex.6.1.25}    that 
$$(I\cap J)^{(k)}=\mathfrak{p}_1^{(k)} \cap \cdots  \cap   \mathfrak{p}_s^{(k)} \cap  \mathfrak{q}_1^{(k)} \cap \cdots  \cap   \mathfrak{q}_t^{(k)} 
 =  I^{(k)} \cap J^{(k)}.$$

\bigskip
\textbf{Case 2.}     $\mathfrak{p}_i \subseteq  \mathfrak{q}_j$ or  $\mathfrak{q}_j \subseteq \mathfrak{p}_i$ for some  $1\leq i \leq s$ and 
 some  $1\leq j \leq t$. We assume that  $\mathfrak{p}_i \subseteq  \mathfrak{q}_j$  for some  $1\leq i \leq s$ and  some  $1\leq j \leq t$. Hence, 
$\mathfrak{q}_j$   can  be removed  from $\mathfrak{p}_1 \cap \cdots \cap  \mathfrak{p}_s \cap  \mathfrak{q}_1 \cap  \cdots \cap  \mathfrak{q}_t$.
 On the other hand, since $\mathfrak{p}_i^{(k)}=\mathfrak{p}_i^k$ and $\mathfrak{q}_j^{(k)}=\mathfrak{q}_j^k$ and $\mathfrak{p}_i \subseteq  \mathfrak{q}_j$, we get  $\mathfrak{p}_i^{(k)}\subseteq  \mathfrak{q}_j^{(k)}$. This means that $\mathfrak{q}_j^{(k)}$ can  be deleted from 
$\mathfrak{p}_1^{(k)} \cap \cdots  \cap   \mathfrak{p}_s^{(k)} \cap  \mathfrak{q}_1^{(k)} \cap \cdots  \cap   \mathfrak{q}_t^{(k)}$, 
as required. The case  $\mathfrak{q}_j \subseteq \mathfrak{p}_i$ can be shown by a similar  argument.   We can therefore  deduce that 
 $(I\cap J)^{(k)}=I^{(k)} \cap J^{(k)}$. This completes the proof. 
  \end{proof}
 

The following proposition  is essential for us to provide  Example   \ref{Exa.NTF.2}. 

 \begin{proposition} \label{Bipartite}
Let   $G$ be  a  bipartite graph with $V(G)=\{x_1, \ldots, x_n\}$ and $I=(u_1, \ldots, u_t, u_{t+1})$ be  its edge ideal. 
 Let $L=(u_1, \ldots, u_t) + (x_{n+1}u_{t+1}) \subset S=K[x_1, \ldots, x_n, x_{n+1}]$. Then $L$ is a 
  normally torsion-free square-free monomial ideal in $S$.   
 \end{proposition}
 
 \begin{proof}
  To simplify the notation,  we put $J:=(u_1, \ldots, u_t)$.  We first claim that $I^rJ^s=I^{r+s} \cap J^s$ for all $r,s \geq 0$. 
   The claim is true for $r=0$ or    $s=0$ or both of them; hence,  fix $r,s \geq 1$.   Since $I=J+(u_{t+1})$, we obtain the following equalities
  \begin{align*}
  I^{r+s}\cap J^s=& (J+(u_{t+1}))^{r+s} \cap J^s\\
  =& (\sum_{\lambda=0}^{r+s} J^{\lambda}u^{r+s-\lambda}_{t+1}) \cap J^s\\
  =& \sum_{\lambda=0}^{r+s} (J^{\lambda}u^{r+s-\lambda}_{t+1} \cap J^s)\\
  =& \sum_{\lambda=0}^{s-1} (J^{\lambda}u^{r+s-\lambda}_{t+1} \cap J^s) +
   \sum_{\lambda=s}^{r+s} (J^{\lambda}u^{r+s-\lambda}_{t+1} \cap J^s).
  \end{align*}
  Let $s\leq \lambda \leq r+s$. Then $J^{\lambda}u^{r+s-\lambda}_{t+1}  \subseteq J^s$, and so we can conclude that 
  $$\sum_{\lambda=s}^{r+s} (J^{\lambda}u^{r+s-\lambda}_{t+1} \cap J^s)= \sum_{\lambda=s}^{r+s} J^{\lambda}u^{r+s-\lambda}_{t+1} 
   =J^s \sum_{\theta=0}^{r} J^{\theta}u^{r- \theta}_{t+1}= J^sI^r.$$
    In what follows, we show that   $J^{\lambda}u^{r+s-\lambda}_{t+1} \cap J^s \subseteq J^su^{r}_{t+1}$ for all $0\leq \lambda \leq s$. 
  For this purpose,   fix $0 \leq \lambda \leq s$. Since  $J$ is a normally torsion-free square-free monomial ideal, we can deduce from 
  Definition \ref{Def.4.3.22}, Fact  \ref{Ex.6.1.25}, and Theorem \ref{Villarreal1} that
   $J^\lambda = J^{(\lambda)}= \bigcap_{\mathfrak{p}\in \mathrm{Min}(J)} \mathfrak{p}^{\lambda}$ 
  and  $J^s = J^{(s)}= \bigcap_{\mathfrak{p}\in \mathrm{Min}(J)} \mathfrak{p}^{s}$. According to  \cite[Exercise 7.9.1(a)]{MRS}, we get the following equalities
  \begin{align*}
   J^{\lambda}u^{r+s-\lambda}_{t+1} \cap J^s =&    \left(u^{r+s-\lambda}_{t+1} \bigcap_{\mathfrak{p}\in \mathrm{Min}(J)} \mathfrak{p}^{\lambda}\right) \cap   \left(\bigcap_{\mathfrak{p}\in \mathrm{Min}(J)} \mathfrak{p}^{s}\right)\\
  =& \left(\bigcap_{\mathfrak{p}\in \mathrm{Min}(J)}(u^{r+s-\lambda}_{t+1}  \mathfrak{p}^{\lambda})\right) \cap 
  \left(\bigcap_{\mathfrak{p}\in \mathrm{Min}(J)} \mathfrak{p}^{s}\right)\\
  =& \bigcap_{\mathfrak{p}\in \mathrm{Min}(J)}\left((u^{r+s-\lambda}_{t+1}  \mathfrak{p}^{\lambda}) \cap 
   \mathfrak{p}^{s}\right).
  \end{align*}
  We also  have $ J^su^{r}_{t+1}= \bigcap_{\mathfrak{p}\in \mathrm{Min}(J)}(u^{r}_{t+1}  \mathfrak{p}^{s})$. Fix $\mathfrak{p}\in \mathrm{Min}(J)$. 
  We show that $(u^{r+s-\lambda}_{t+1}  \mathfrak{p}^{\lambda}) \cap    \mathfrak{p}^{s} \subseteq u^{r}_{t+1}  \mathfrak{p}^{s}$. 
     Pick a monomial $v\in \mathcal{G}((u^{r+s-\lambda}_{t+1}  \mathfrak{p}^{\lambda}) \cap    \mathfrak{p}^{s})$. It follows from 
  \cite[Proposition 1.2.1]{HH1} that  $v=\mathrm{lcm}(u^{r+s-\lambda}_{t+1}f,g)$, 
    where $f\in \mathcal{G}({\mathfrak{p}}^{\lambda})$ and $g\in  \mathcal{G}({\mathfrak{p}}^{s})$.  On account of   $f\in \mathcal{G}({\mathfrak{p}}^{\lambda})$ (respectively, $g\in  \mathcal{G}({\mathfrak{p}}^{s})$), we can write $f=\prod_{i=1}^nx^{a_i}_i$ 
    (respectively,   $g=\prod_{i=1}^nx^{b_i}_i$) such that  $a_i\geq 0$ for all $i=1, \ldots, n$, $\sum_{i=1}^n a_i=\lambda$, and $x_i\in \mathfrak{p}$ when 
    $a_i>0$     (respectively, $b_i\geq 0$ for all $i=1, \ldots, n$, $\sum_{i=1}^n b_i=s$, and $x_i\in \mathfrak{p}$ when $b_i>0$). 
      Let  $u_{t+1}=x_\alpha x_\beta$ with $1\leq  \alpha \neq \beta \leq n$.  Since $v=\mathrm{lcm}(u^{r+s-\lambda}_{t+1}f,g)$, we can derive that    
  $$v=\prod_{i=1, i\notin\{\alpha, \beta\}}^{n} x_{i}^{\max\{a_i,b_i\}} x_{\alpha}^{\max\{r+s-\lambda +a_\alpha, b_\alpha\}}  
   x_{\beta}^{\max\{r+s-\lambda + a_\beta,b_\beta\}}.$$
   Due to $\deg_{x_{\alpha}} v= \max\{r+s-\lambda +a_\alpha, b_\alpha\}$ and $\deg_{x_{\beta}} v=\max\{r+s-\lambda + a_\beta,b_\beta\}$, we obtain 
      $\deg_{x_{\alpha}} v \geq r+s-\lambda +a_\alpha \geq r$ and $\deg_{x_{\beta}} v \geq r+s-\lambda + a_\beta \geq r$. As 
   $\sum_{i=1}^n a_i=\lambda$ and $\deg_{x_i} v =\max\{a_i, b_i\}$ for all $i\in \{1, \ldots, n\}\setminus \{\alpha, \beta\}$, we can conclude that   
   $$\sum_{i=1}^n\deg_{x_i} v \geq \sum_{i=1, i\notin\{\alpha, \beta\}}^na_i +r+s-\lambda +a_\alpha + r+s-\lambda + a_\beta \geq s+2r.$$
   It follows from   $\deg_{x_{\alpha}} v \geq r$, $\deg_{x_{\beta}} v \geq  r$, and $\sum_{i=1}^n\deg_{x_i} v \geq s+2r$  that $v\in u^{r}_{t+1}  \mathfrak{p}^{s}$. This leads to  $(u^{r+s-\lambda}_{t+1}  \mathfrak{p}^{\lambda}) \cap    \mathfrak{p}^{s} \subseteq u^{r}_{t+1}  \mathfrak{p}^{s}$ 
 for all   $\mathfrak{p}\in \mathrm{Min}(J)$,  and hence  $J^{\lambda}u^{r+s-\lambda}_{t+1} \cap J^s \subseteq J^su^{r}_{t+1}$ for all $0\leq \lambda \leq s$. Accordingly, we have    $I^rJ^s=I^{r+s} \cap J^s$ for all $r,s \geq 0$, as claimed. 
 
 To complete the proof, by  Theorem \ref{Villarreal1}, it is sufficient to demonstrate that  $L^k=L^{(k)}$ for all  $k\geq 1$. To do this, fix $k\geq 1$.  
 Since     $L=J+x_{n+1}u_{t+1}S$, we can  deduce from Fact \ref{fact1} that $L=(J, x_{n+1}) \cap (J, u_{t+1})=J+x_{n+1}I$. 
 It follows from  Proposition \ref{Intersection}   that $L^{(k)}=(J, x_{n+1})^{(k)} \cap (J, u_{t+1})^{(k)}$. It is well-known that the edge ideal of any bipartite 
 graph is normally torsion-free, and so $I=J+ (u_{t+1})$ is normally torsion-free. In addition, if we delete an edge from  a bipartite graph, then  the new graph is still bipartite, and hence $J$ is normally torsion-free. Due to $x_{n+1} \notin \mathrm{supp}(J)$, Theorem \ref{NTF.Th.2.5}  yields that $(J, x_{n+1})$ is normally 
 torsion-free. We therefore get $(J, u_{t+1})^{(k)}= (J, u_{t+1})^{k}$  and  $(J, x_{n+1})^{(k)}=(J, x_{n+1})^{k}$. This leads to  
  $L^{(k)}=(J, x_{n+1})^{k} \cap (J, u_{t+1})^{k}$.  From  the binomial theorem and Fact \ref{fact1}, we have the following equalities 
 \begin{align*}
  (J, x_{n+1})^k  \cap (J, u_{t+1})^k=& (J, x_{n+1})^k  \cap I^k=   \left(\sum_{\alpha=0}^{k}x_{n+1}^{k-\alpha}J^{\alpha}\right) \cap  I^k\\
  = &\sum_{\alpha=0}^{k} (x_{n+1}^{k-\alpha}J^{\alpha} \cap  I^k) =  \sum_{\alpha=0}^{k} x_{n+1}^{k-\alpha}J^{\alpha}   I^{k-\alpha} \\
  =& (J+ x_{n+1}I)^k\\
  =& L^k.
  \end{align*}
We thus get  $L^{k}=L^{(k)}$, and  Theorem \ref{Villarreal1} implies that $L$ is normally torsion-free, as desired.  
 \end{proof}

 
 It is natural to ask that can we  generalize Proposition \ref{Bipartite} to any square-free monomial ideal? In fact, assume that 
   $I=(u_1, \ldots, u_t, u_{t+1})$ is   a normally torsion-free square-free monomial ideal such that $\mathrm{supp}(I) \subseteq \{x_1, \ldots, x_n\}$ and 
  $L=(u_1, \ldots, u_t) + (x_{n+1}u_{t+1}) \subset S=K[x_1, \ldots, x_n, x_{n+1}]$. Then it is possible $L$ is not normally torsion-free. We provide such a counterexample in     Example \ref{Exa.NTF.1}.  For this purpose, we utilize Theorem \ref{Th.1} (when $\ell=2$).

 \begin{example} \label{Exa.NTF.1}
 {\em 
 Let $I=(x_3x_6, x_2x_5, x_1x_4, x_1x_5x_6, x_2x_4x_6, x_3x_4x_5, x_1x_2x_3)$ be a square-free monomial ideal in $R=K[x_1, \ldots, x_6]$. 
  We first show that $I$ is normally torsion-free.  To do this, our strategy is to use  Theorem \ref{Th.1}. Using \textit{Macaulay2} \cite{GS}, we have 
  $$\mathrm{Min}(I)=\mathrm{Ass}(I)=\{(x_3,x_2,x_1), (x_6,x_5,x_1), (x_6,x_4,x_2), (x_5,x_4,x_3)\}.$$
 Put $v:=x_3x_6$. It is not hard to check that $v\in \mathfrak{p}\setminus \mathfrak{p}^2$ for any $\mathfrak{p}\in \mathrm{Min}(I)$. 
 We demonstrate that $I\setminus x_3$ and $I\setminus x_6$ are normally torsion-free in two steps.
 
 \bigskip
\textbf{Step  1.}   We have $I\setminus x_3=(x_2x_5,x_1x_4,x_1x_5x_6,x_2x_4x_6)$, and so    
$$\mathrm{Min}(I\setminus x_3)=\mathrm{Ass}(I\setminus x_3)=\{(x_2,x_1), (x_5,x_4), (x_6,x_5,x_1), (x_6,x_4,x_2)\}.$$
To show  $I\setminus x_3$ is normally torsion-free, we use Theorem \ref{Th.1}. Set $v_1:=x_2x_5$ and $A:=I\setminus x_3$. 
 It is routine to check $v_1\in \mathfrak{p}\setminus \mathfrak{p}^2$ for any $\mathfrak{p}\in \mathrm{Min}(A)$. 
Furthermore, we have $A\setminus x_2=x_1(x_4, x_5x_6)$ and $A\setminus x_5=x_4(x_1, x_2x_6)$. According to 
 Theorem \ref{NTF.Th.2.5}  and  Lemma \ref{NTF.Lem.3.12}, both $A\setminus x_2$ and $A\setminus x_5$ are normally torsion-free. We thus obtain $I\setminus x_3$ is normally torsion-free. 

\bigskip
\textbf{Step  2.}   We have $I\setminus x_6=(x_2x_5, x_1x_4, x_3x_4x_5, x_1x_2x_3)$, and hence  
 $$\mathrm{Min}(I\setminus x_6)=\mathrm{Ass}(I\setminus x_6)=\{(x_5,x_1),(x_4,x_2),(x_3,x_2,x_1),(x_5,x_4,x_3)\}.$$
By using  Theorem \ref{Th.1}, we prove that   $I\setminus x_6$ is normally torsion-free.  Take  $v_2:=x_2x_5$ and $B:=I\setminus x_6$. 
 It is easy to see that  $v_2\in \mathfrak{p}\setminus \mathfrak{p}^2$ for any $\mathfrak{p}\in \mathrm{Min}(B)$. 
Moreover, we get  $B\setminus x_2=x_4(x_1, x_3x_5)$ and $B\setminus x_5=x_1(x_4,x_2x_3)$. In the light of  
  Theorem \ref{NTF.Th.2.5}  and  Lemma \ref{NTF.Lem.3.12}, both $B\setminus x_2$ and $B\setminus x_5$ are normally torsion-free. We therefore deduce that  $I\setminus x_6$ is normally torsion-free. 

Consequently, $I$ is normally torsion-free. Now, consider the following square-free monomial ideal in the polynomial ring $S=K[x_1, \ldots, x_7]$,
$$L:=(x_3x_6x_7, x_2x_5, x_1x_4, x_1x_5x_6, x_2x_4x_6, x_3x_4x_5, x_1x_2x_3).$$  
We can write $L=(x_3x_6x_7) +(x_2x_5, x_1x_4, x_1x_5x_6, x_2x_4x_6, x_3x_4x_5, x_1x_2x_3).$ 
Using \textit{Macaulay2} \cite{GS}, we obtain $(x_7,x_5,x_4,x_2,x_1) \in \mathrm{Ass}(S/L^2) \setminus \mathrm{Ass}(S/L)$. This means that $L$ is not normally torsion-free. 
}
 \end{example}


  Let $L=x_iI+x_jJ\subset R=K[x_1, \ldots, x_n]$ with $1\leq i\neq j \leq n$ be a square-free monomial ideal such that $\mathrm{gcd}(x_j, u)=1$  and 
   $\mathrm{gcd}(x_i, v)=1$   for all $u\in \mathcal{G}(I)$ and  $v\in \mathcal{G}(J)$. Also,    let  $I$, $J$, $I+J$, and $I+x_jJ$  be  normally 
   torsion-free  in $R$.   The following example shows that it is possible both $x_iI+J$ and $L$ are  not normally torsion-free in $R$. 
  In particular, Example \ref{Exa.NTF.2} tells us that the  hypotheses in Theorem \ref{Linear-Combination} are best possible.
 
 \begin{example}  \label{Exa.NTF.2}
 {\em 
 Suppose that   $L \subset  R=K[x_1, \ldots, x_8]$ is   the  following square-free monomial ideal 
 \begin{align*}
 L:=(& x_1x_2x_3, x_2x_3x_4, x_3x_4x_5, x_3x_7x_8, x_1x_3x_8, x_5x_3x_7, x_2x_6x_7, \\
&  x_4x_5x_6,  x_5x_6x_7,x_6x_7x_8, x_1x_2x_6).
 \end{align*}
   Let  square-free monomial ideals  $I:=(x_1x_2,x_2x_4,x_4x_5,x_7x_8,x_1x_8, x_5x_7)$ and $J:=(x_4x_5,x_5x_7,x_7x_8,x_1x_2, x_2x_7)$. Then 
  $L=x_3I+x_6J$.
 It is routine to see that  $\mathrm{gcd}(x_6, u)=1$  and   $\mathrm{gcd}(x_3, v)=1$   for all $u\in \mathcal{G}(I)$ and  $v\in \mathcal{G}(J)$. 
 Since $I$, $J$, and $I+J$  can be seen as the edge ideals of three bipartite graphs, we thus  conclude that all  $I$,  $J$, and $I+J$ 
  are normally torsion-free in $R$. 
  
  We now verify that $I+x_6J=(x_7x_8,x_1x_8,x_5x_7, x_4x_5,x_2x_4,x_1x_2,x_2x_6x_7)$ is normally torsion-free. 
 Set $M:=(x_7x_8,x_1x_8,x_5x_7, x_4x_5,x_2x_4,x_1x_2,x_2x_7)$. Thanks to  $M$ can be viewed as the edge ideal of a bipartite graph $G$ 
 with $V(G)=\{x_1, x_2, x_4, x_5, x_7, x_8\}$, it follows from Proposition \ref{Bipartite} that
$$
 I+x_6J=(x_7x_8,x_1x_8,x_5x_7, x_4x_5,x_2x_4,x_1x_2) + x_6(x_2x_7),
 $$
 is normally torsion-free as well. 
On the other hand, using \textit{Macaulay2} \cite{GS}, we  obtain  
 $(x_1, x_2, x_3, x_5, x_7)\in \mathrm{Ass}(R/(x_3I+J)^2)\setminus \mathrm{Ass}(R/(x_3I+J))$  and  
$(x_1, x_4, x_5, x_6, x_8)\in \mathrm{Ass}(R/L^2)\setminus \mathrm{Ass}(R/L)$. In other words, both  $x_3I+J$ and $L$ 
 are  not normally torsion-free in $R$.
}
\end{example}

 
  We can now formulate the  main result of this section in the next  theorem. In fact, Theorem \ref{Linear-Combination} gives  a necessary and 
  sufficient condition to determine the normally torsion-freeness  of a linear combination of two    normally torsion-free square-free monomial   ideals.

 \begin{theorem} \label{Linear-Combination}
  Let $L=x_iI+x_jJ\subset R=K[x_1, \ldots, x_n]$ with $1\leq i\neq j \leq n$ be a square-free monomial ideal such that $\mathrm{gcd}(x_j, u)=1$  and   $\mathrm{gcd}(x_i, v)=1$   for all $u\in \mathcal{G}(I)$ and  $v\in \mathcal{G}(J)$. Then $L$ is normally torsion-free if and only if 
     $x_iI+J$ and $I+x_jJ$   are  normally torsion-free  in $R$.     
 \end{theorem}

  \begin{proof}
  ($\Rightarrow$) Let $L$ be normally torsion-free. It follows from  Theorem \ref{NTF.Th.3.19}  that both $L/ x_i$  (the contraction of $L$ at $x_i$) and
   $L/ x_j$  (the contraction of $L$ at $x_j$) are normally torsion-free. This implies
   that $L/x_i= I+x_jJ$ and $L/x_j=x_iI+J$ are normally torsion-free.
   
   ($\Leftarrow$) Assume that  $I+x_jJ$ and $x_iI+J$ are normally torsion-free.         
  Due to $I+x_jJ$ is  normally torsion-free, we can deduce from Theorems  \ref{NTF.Th.3.19} and \ref{NTF.Th.3.21}
   that  $(I+x_jJ)/x_j$ (the contraction of $I+x_jJ$ at $x_j$)    and 
  $(I+x_jJ)\setminus x_j$ (the deletion of $I+x_jJ$ at $x_j$)   are normally torsion-free. Hence, $I+J$ and $I$ are normally torsion-free.  In addition, 
  Lemma \ref{NTF.Lem.3.12} implies that  $x_iI$ is normally torsion-free, and by virtue of  Theorem \ref{NTF.Th.2.5},  we get $(x_iI, x_j)$ and $(I, x_j)$ are  normally torsion-free. 
  In view of  Theorem \ref{Villarreal1}, we derive   $(I+J)^{k}=(I+J)^{(k)}$, $(I, x_j)^k=(I, x_j)^{(k)}$,  $(x_iI, x_j)^k=(x_iI, x_j)^{(k)}$, and 
  $(x_iI+J)^k=(x_iI+J)^{(k)}$.   Because  $L$ is square-free, we must have $\mathrm{gcd}(x_j, v)=1$   for all $v\in \mathcal{G}(J)$, and  
  since $\mathrm{gcd}(x_j, u)=1$    for all $u\in \mathcal{G}(I)$,  by using Fact \ref{fact1}, we obtain $L=(x_iI, x_j) \cap (x_iI, J)$ and $I+x_jJ=(I, x_j) \cap (I,J).$  
   Fix $k\geq 1$.  It follows from  Proposition \ref{Intersection}  that 
   \begin{equation}
   L^{(k)}=(x_iI, x_j)^{(k)} \cap (x_iI, J)^{(k)}=(x_iI, x_j)^{k} \cap (x_iI, J)^{k}, \label{eq.6.3}
   \end{equation}
    and 
 \begin{equation}
 (I+x_jJ)^{(k)}=(I, x_j)^{(k)} \cap (I,J)^{(k)}=(I, x_j)^{k} \cap (I,J)^{k}. \label{eq.6.4}
 \end{equation}
  
    Hence, we obtain  $(I+x_jJ)^{k}=(I, x_j)^{k} \cap (I,J)^{k}.$
    Using the binomial theorem and Fact \ref{fact1}, we get 
  \begin{align*}
  (I, x_j)^k  \cap (I,J)^k=& \left(I^k+ \sum_{\alpha=0}^{k-1}x_j^{k-\alpha}I^{\alpha}\right) \cap 
  \left(I^k+ \sum_{\beta=0}^{k-1} J^{k-\beta} I^{\beta}\right) \\
  =& I^k + \sum_{\alpha=0}^{k-1}  \sum_{\beta=0}^{k-1}  x_j^{k-\alpha}I^{\alpha}  \cap    J^{k-\beta} I^{\beta}\\
  =&  \sum_{\alpha=0}^{k-1}  \sum_{\beta=0}^{\alpha-1}  x_j^{k-\alpha}I^{\alpha}  \cap    J^{k-\beta} I^{\beta} \\
  +&  \left(I^k + \sum_{\alpha=0}^{k-1}    x_j^{k-\alpha}I^{\alpha}  \cap    J^{k-\alpha} I^{\alpha}\right)\\
  +&  \sum_{\alpha=0}^{k-1}  \sum_{\beta=\alpha +1}^{k-1}  x_j^{k-\alpha}I^{\alpha}  \cap    J^{k-\beta} I^{\beta}\\
    =&\sum_{\alpha=0}^{k-1}  \sum_{\beta=0}^{\alpha-1}  x_j^{k-\alpha}I^{\alpha}  \cap    J^{k-\beta} I^{\beta}   + (I+x_jJ)^k\\
   +& \sum_{\alpha=0}^{k-1}  \sum_{\beta=\alpha +1}^{k-1}  x_j^{k-\alpha} J^{k-\beta} I^{\beta}. 
     \end{align*}
  Furthermore, for each $\alpha+1 \leq \beta \leq  k-1$  and $0 \leq \alpha \leq k-1$, we have 
  $$x_j^{k-\alpha}J^{k-\beta}I^{\beta} = x_j^{\beta - \alpha} (x_j^{k-\beta}J^{k-\beta}I^{\beta}) \subseteq (I+x_jJ)^k.$$
  This gives rise to $\sum_{\alpha=0}^{k-1}  \sum_{\beta=\alpha +1}^{k-1}  x_j^{k-\alpha} J^{k-\beta} I^{\beta} \subseteq (I+x_jJ)^k$, and so 
  $$\sum_{\alpha=0}^{k-1} x_j^{k-\alpha} \left(I^{\alpha}  \cap \sum_{\beta=0}^{\alpha-1} J^{k-\beta} I^{\beta}\right) \subseteq   (I+x_jJ)^k.$$
   Accordingly, for all   $\alpha=0, \ldots, k-1$,  we can deduce   the following inclusion 
   \begin{equation}
    I^\alpha \cap \sum_{\beta =0}^{\alpha-1} J^{k-\beta}I^{\beta} \subseteq J^{k-\alpha} I^\alpha. \label{eq.6.5}
   \end{equation}
    From  the binomial theorem,  Fact \ref{fact1}, and   (\ref{eq.6.3}), we  get   the following 
    \begin{align*}
      L^{(k)}=& (x_iI, x_j)^{k} \cap (x_iI, J)^{k}\\
      =& \left((x_iI)^k+\sum_{\alpha=0}^{k-1}x_j^{k-\alpha}(x_iI)^{\alpha} \right) \cap
       \left((x_iI)^k+\sum_{\beta=0}^{k-1} J^{k-\beta}(x_iI)^{\beta} \right)\\
       =& (x_iI)^k+\sum_{\alpha=0}^{k-1} \sum_{\beta=0}^{k-1}  x_j^{k-\alpha}(x_iI)^{\alpha} 
       \cap        J^{k-\beta}(x_iI)^{\beta} \\
       =& \sum_{\alpha=0}^{k-1} \sum_{\beta=0}^{\alpha-1}  x_j^{k-\alpha}(x_iI)^{\alpha} 
       \cap        J^{k-\beta}(x_iI)^{\beta} \\
       +& \left((x_iI)^k+\sum_{\alpha=0}^{k-1}   x_j^{k-\alpha}(x_iI)^{\alpha}    \cap     J^{k-\alpha}(x_iI)^{\alpha}\right) \\
       +& \sum_{\alpha=0}^{k-1} \sum_{\beta=\alpha+1}^{k-1}  x_j^{k-\alpha}(x_iI)^{\alpha} 
       \cap        J^{k-\beta}(x_iI)^{\beta} \\
       =& \sum_{\alpha=0}^{k-1} \sum_{\beta=0}^{\alpha-1}  x_j^{k-\alpha}(x_iI)^{\alpha} 
       \cap        J^{k-\beta}(x_iI)^{\beta} + (x_iI+x_jJ)^k \\
       +& \sum_{\alpha=0}^{k-1} \sum_{\beta=\alpha+1}^{k-1}  x_j^{k-\alpha}     J^{k-\beta}(x_iI)^{\beta}. \\
    \end{align*} 
  In addition, for each $\alpha+1 \leq \beta \leq  k-1$  and $0 \leq \alpha \leq k-1$, we have 
    $$x_j^{k-\alpha}J^{k-\beta}(x_iI)^{\beta} = x_j^{\beta - \alpha} (x_j^{k-\beta}J^{k-\beta}(x_iI)^{\beta}) \subseteq (x_iI+x_jJ)^k.$$ 
   This leads to $\sum_{\alpha=0}^{k-1} \sum_{\beta=\alpha+1}^{k-1}  x_j^{k-\alpha}     J^{k-\beta}(x_iI)^{\beta} \subseteq (x_iI+x_jJ)^k$. Using 
    (\ref{eq.6.5}), we can derive the following 
   \begin{align*}
    \sum_{\beta=0}^{\alpha-1} x_j^{k-\alpha}(x_iI)^{\alpha}  \cap   J^{k-\beta}(x_iI)^{\beta}
   = &   x_j^{k-\alpha}x_i^{\alpha}   \left(I^{\alpha}   \cap  \sum_{\beta=0}^{\alpha-1}  J^{k-\beta} I^{\beta}\right)\\
     \subseteq &       x_j^{k-\alpha}x_i^{\alpha}  J^{k-\alpha} I^\alpha\\
    \subseteq  & (x_iI+x_jJ)^k.
     \end{align*}
  Consequently, $L^{(k)}=(x_iI+x_jJ)^k$, and so $L^{(k)}=L^k$. One can conclude from Theorem \ref{Villarreal1} that $L$ is normally torsion-free, as claimed.     
     \end{proof}
     
We conclude this paper  with the following example, which  shows  that how one can use Theorem \ref{Linear-Combination} to inspect  
the normally torsion-freeness of   square-free monomial ideals. 
  
  \begin{example}  \label{Exam. Linear. C8} 
  {\em 
  Let $C_8=(V(C_8), E(C_8))$ denote the cycle graph with vertex set $V(C_8)=\{x_1, \ldots, x_8\}$ and the following 
 edge set $$E(C_8)=\{\{x_i, x_{i+1}\} : i=1, \ldots, 7\} \cup \{\{x_1, x_8\}\}.$$ 
 Now, consider the following square-free monomial ideal in  $R=K[x_1, \ldots, x_8]$,
 \begin{align*}
 L=(&x_1x_2x_3x_4, x_2x_3x_4x_5, x_3x_4x_5x_6, x_4x_5x_6x_7, x_5x_6x_7x_8, x_6x_7x_8x_1, \\
 & x_7x_8x_1x_2, x_8x_1x_2x_3).
 \end{align*}
 In fact, $L$ is  the edge ideal of  the hypergraph $\mathcal{H}_3(C_8)$, where $\mathcal{H}_3(C_8)$ denotes  the  $4$-uniform hypergraph on the vertex set  $V(C_8)$, more information can be found in \cite{HN}.   It has already been shown in \cite[Proposition 3.4]{HN} that  $\mathcal{H}_3(C_8)$ is Mengerian, and so $L$ is normally torsion-free. Now,  our aim is to re-prove this  fact by using Theorem  \ref{Linear-Combination}.  Let 
 $I:=(x_1x_2x_3, x_2x_3 x_5, x_3x_5x_6, x_5x_6x_7)$ and $J:=(x_5x_6x_7, x_6x_7x_1, x_7x_1x_2, x_1x_2x_3)$. It is easy to check that $L=x_4I+x_8J$. 
 To establish the normally torsion-freeness of $L$, by Theorem \ref{Linear-Combination}, it is sufficient to show that both $x_4I+J$ and $I+x_8J$ are normally
 torsion-free. We accomplish this in two steps:
 
  \bigskip
\textbf{Step  1.} Our aim is to show that the following square-free monomial ideal 
 $$x_4I+J=(x_6x_7x_1,x_5x_6x_7,x_1x_2x_7,x_1x_2x_3,x_3x_4x_5x_6,x_2x_3x_4x_5),$$ is normally torsion-free. To do this, we consider 
  $A=(x_6x_1,x_5x_6,x_1x_2)$ and $B:=(x_1x_2,x_4x_5x_6,x_2x_4x_5)$. It is clear that $x_4I+J=x_7A+x_3B$.  Due to   Theorem \ref{Linear-Combination}, 
 we must verify that $x_7A + B$ and $A+x_3B$ are normally torsion-free. We first prove the normally torsion-freeness of $x_7A+B$. Put $A_1:=(x_7x_1,x_5x_7,x_4x_5)$ and  $B_1:=(x_1, x_4x_5)$. Obviously,  we have $x_7A+B=x_6A_1+x_2B_1$. So, we have to prove that
  $x_6A_1+B_1=(x_5x_6x_7, x_4x_5, x_1)$ and 
 $A_1+x_2B_1=(x_7x_1,x_5x_7,x_4x_5, x_1x_2)$ are normally torsion-free. Since    $A_1+x_2B_1=(x_7x_1,x_5x_7,x_4x_5, x_1x_2)$
  can be viewed as the edge ideal of a  bipartite graph, this implies that $A_1+x_2B_1$ is normally torsion-free. In addition, based on Theorem \ref{NTF.Th.2.5}, 
  $(x_6x_7, x_4)$ is normally torsion-free, and so $x_5(x_6x_7, x_4)$ is normally torsion-free. Once again,  Theorem \ref{NTF.Th.2.5}  implies that 
  $x_6A_1+B_1=(x_5x_6x_7, x_4x_5, x_1)$ is normally torsion-free. 
   Hence, we deduce that $x_7A + B$ is normally torsion-free. We now prove  that  $A+x_3B=(x_6x_1, x_5x_6, x_1x_2, x_2x_3x_4x_5)$ is  normally torsion-free. Set $A_2:=(x_1, x_5)$ and $B_2:=(x_1, x_3x_4x_5)$. 
 Hence,   $A+x_3B=x_6A_2+ x_2B_2$. It follows from Theorem \ref{NTF.Th.2.5}  that $A_2+x_2B_2=(x_1, x_5)$ is normally torsion-free.
  By  Theorem \ref{NTF.Th.2.5}, $(x_6, x_3x_4)$ is normally torsion-free, and so $x_5(x_6, x_3x_4)$ is normally torsion-free. This gives that 
  $x_6A_2+B_2=(x_5x_6, x_3x_4x_5, x_1)$ is  normally torsion-free. Consequently, $A + x_3B$ is normally torsion-free.

\bigskip
\textbf{Step  2.} We want to verify that the following square-free monomial ideal 
$$I+x_8J:=(x_1x_2x_3, x_2x_3 x_5, x_3x_5x_6, x_5x_6x_7,  x_8x_6x_7x_1, x_8x_7x_1x_2),$$
is normally torsion-free.  For this purpose, set $C:=(x_1x_2, x_2 x_5, x_5x_6)$ and $D:=(x_5x_6,  x_8x_6x_1, x_8x_1x_2)$. Since $I+x_8J=x_3C+x_7D$,  by  Theorem \ref{Linear-Combination}, we have to establish $x_3C+D=(x_1x_2x_3, x_2x_3x_5, x_5x_6, x_1x_6x_8, x_1x_2x_8)$ 
and $C+x_7D=(x_1x_2, x_2 x_5, x_5x_6, x_1x_6x_7x_8)$ are normally torsion-free. To show the normally torsion-freeness 
of $x_3C+D$, put $C_1:=(x_1x_3, x_3x_5, x_1x_8)$ and $D_1:=(x_5, x_1x_8)$. Clearly, we have $x_3C+D=x_2C_1+x_6D_1$. Based on Theorem \ref{Linear-Combination}, we must prove that $x_2C_1+D_1$ and $C_1+x_6D_1$ are normally torsion-free. Since $C_1+x_6D_1=(x_1x_3, x_3x_5, x_1x_8, x_5x_6)$ can be viewed as the edge ideal of a bipartite graph, this yields that $C_1+x_6D_1$ is normally torsion-free. Moreover,  Theorem \ref{NTF.Th.2.5}  implies that 
 $x_2C_1+D_1=(x_1x_2x_3, x_1x_8, x_5)$ is normally torsion-free, and so $x_3C+D=x_2C_1+x_6D_1$ is normally torsion-free. Here, we focus on 
$C+x_7D=(x_1x_2, x_2 x_5, x_5x_6, x_1x_6x_7x_8)$. Take $C_2:=(x_2, x_6x_7z_8)$ and $D_2:=(x_2,x_6)$. 
  Thus,  $C+x_7D=x_1C_2+x_5D_2$, and by  Theorem \ref{Linear-Combination}, one has  to explore the normally torsion-freeness of   $x_1C_2+D_2$  
  and  $C_2+x_5D_2$.  Since $(x_5x_6, x_6x_7x_8)=x_6(x_5, x_7x_8)$ is normally torsion-free, we can conclude from Theorem \ref{NTF.Th.2.5}  that $C_2+x_5D_2=(x_2, x_5x_6, x_6x_7x_8)$ is normally torsion-free. Thanks to $x_1C_2+D_2=(x_2, x_6)$ is  normally torsion-free, it follows from    Theorem \ref{Linear-Combination}  that  $C+x_7D=x_1C_2+x_5D_2$ is normally torsion-free.  This completes our argument. 
}
  \end{example}


 \vspace{1cm}
\centerline{\bf  The conflict of interest and data availability statement}

\vspace{4mm}
We hereby declare  that this  manuscript has no associated  data and also there is no conflict of interest in this manuscript.


\bigskip
\centerline{\bf Acknowledgments}
\vspace{4mm}
 
 The authors would like to thank Professor Melvin Hochster from University of Michigan for some helpful comments during the preparation of this paper. 



 
 \vspace{1cm}
  \textbf{Authors’ addresses}
 
 \bigskip
 \begin{itemize}
 \item[\textbf{Alain Bretto}]:  Universit\'e de Caen, GREYC CNRS UMR-6072, Campus II, Bd Marechal Juin BP 5186, 14032 Caen cedex, France\\
 Email address: alain.bretto@unicaen.fr \\
 ORCID:  https://orcid.org/0000-0003-2619-8070  \\
\item[\textbf{ Mehrdad Nasernejad}]: Univ. Artois, UR 2462, Laboratoire de Math\'{e}matique de  Lens (LML),  F-62300 Lens, France, ~~~ and \\
Universit\'e de Caen, GREYC CNRS UMR-6072, Campus II,  Bd Marechal Juin BP 5186, 14032 Caen cedex, France\\
Email address: m$\_$nasernejad@yahoo.com\\ 
ORCID: https://orcid.org/0000-0003-1073-1934\\
\item[\textbf{Jonathan Toledo}]: Tecnol\'o{g}ico Nacional de M\'e{x}ico, Instituto Tecnol\'o{g}ico Del Valle de Etla, Abasolo S/N, Barrio Del Agua Buena, 
Santiago Suchilquitongo, 68230, Oaxaca, M\'e{x}ico\\
E-mail address:  jonathan.tt@itvalletla.edu.mx\\
ORCID: https://orcid.org/0000-0003-3274-1367
  \end{itemize}


\end{document}